\documentclass[11pt,  oneside,  a4paper]{amsart}
\usepackage{amssymb}

\usepackage{amsmath}

\usepackage{amscd,amsopn}

\usepackage{graphicx}
\usepackage{amsmath,  amsfonts,  amssymb,  amsthm}
\usepackage{xypic}
\usepackage{verbatim}

\usepackage{mathrsfs}

\usepackage{stmaryrd}

\usepackage{enumitem}

\usepackage[all]{xy}

\usepackage{textcomp}

\usepackage{amsbsy}

\usepackage{xcolor}

\usepackage{etex}
\usepackage{hyperref}
\hypersetup{
  colorlinks = true,
  linkcolor  = black
}

\usepackage[alphabetic,backrefs,msc-links]{amsrefs}

\usepackage{tikz}
\usepackage{caption}
\usetikzlibrary{decorations.markings}

\DeclareMathOperator{\degg}{deg}

\DeclareMathOperator{\Homm}{Hom}

\DeclareMathOperator{\THI}{THI}

\DeclareMathOperator{\muu}{\mu^{orb}}

\DeclareMathOperator{\Imm}{Im}
\DeclareMathOperator{\spann}{span}
\DeclareMathOperator{\cone}{cone}
\DeclareMathOperator{\rank}{rank}

\newtheorem{THE}{Theorem}[section]

\newtheorem{thm-defn}[THE]{Theorem/Definition}
\newtheorem{LE}[THE]{Lemma}
\newtheorem{PR}[THE]{Proposition}
\newtheorem{COR}[THE]{Corollary}

\theoremstyle{definition}
\newtheorem{DEF}[THE]{Definition}

\newtheorem{eg}[THE]{Example}

\theoremstyle{remark}
\newtheorem*{rmk}{Remark}

\numberwithin{equation}{section}

\author{Yi Xie}
\title{$\mathfrak{sl}(3)$ Khovanov module and the detection of the planar theta-graph}
\date{}

\newcommand{\Address}{{
  \bigskip
  \footnotesize
   Yi Xie, \textsc{Simons Center for Geometry and Physics, State University of New York,
  Stony Brook, NY 11794}\par\nopagebreak
  \textit{E-mail address}: \texttt{yxie@scgp.stonybrook.edu}
}}

\begin{document}
\maketitle

\begin{abstract}
We introduce two invariants called $\mathfrak{sl}(3)$ Khovanov module and pointed $\mathfrak{sl}(3)$ Khovanov homology
for spatial webs (bipartite trivalent graphs). 
Those invariants are related to 
Kronheimer-Mrowka's instanton invariants $J^\sharp$ and $I^\sharp$ for spatial webs by two spectral sequences.
As an application of the spectral sequences, we prove that
$\mathfrak{sl}(3)$ Khovanov module and pointed $\mathfrak{sl}(3)$ Khovanov homology both detect the planar theta graph.
\end{abstract}

\section{Introduction}
In \cite{Kh-Jones}, Khovanov constructed a link homology which categorifies the quantum $\mathfrak{sl}(2)$ link invariant:
the Jones polynomial.   
Later, $\mathfrak{sl}(3)$ and $\mathfrak{sl}(n)$ link homologies were introduced in \cite{Kh-sl3} and \cite{KhR}.

To define the $\mathfrak{sl}(3)$ Khovanov homology of a link $L$ (in $\mathbb{R}^3$),
we need to pick a diagram for $L$ and resolve the crossings
 by resolutions shown in Figure \ref{+01} and Figure \ref{-01}. 
A resolved diagram could be singular. 
In general it is a planar trivalent graph rather than a collection of circles which is the situation of the 
$\mathfrak{sl}(2)$ Khovanov homology. 
A chain complex is defined by assigning abelian groups to those planar trivalent graphs and defining a differential
on the direct sum of those abelian groups. The $\mathfrak{sl}(3)$ Khovanov homology is the homology of the chain complex, which
is a bi-graded abelian group.

The definition of the chain complex depends on the choice of the diagram of the link $L$. 
To obtain a well-defined link invariant, we need to show
that the homology does not depend on the choice of the diagram.
Since any two diagrams of the same link are connected by a 
sequence of Reidemeister moves of three types, it suffices to show that the homology of the chain complex
is invariant under three types of moves. 

In this paper, we consider more general objects: oriented webs embedded in $\mathbb{R}^3$ (see Definition \ref{web}).
Given an oriented web $\Gamma$ in $\mathbb{R}^3$, we can still pick a diagram $D$ for it. 
Resolving the crossings by Figure \ref{+01} or Figure \ref{-01}, we could obtain planar webs. By assigning
 abelian groups
to those planar webs and define a differential properly as in the link case, we obtain a chain complex $F(D)$.
We define the $\mathfrak{sl}(3)$ Khovanov homology for $\Gamma$ to be the homology of this chain complex.
Again the definition depends on the choice of the diagram. 
Any two diagrams of $\Gamma$ 
are connected by a sequence of \emph{five} types of Reidemeister moves according to \cite{Kau-graph,Kau-graph2}.
In addition to the three types of Reidemeister moves in the link case, 
there are two new Reidemeister moves shown in Figures \ref{RIV} and \ref{RV}. To show the homology is a well-defined invariant,
we need to study how it changes under the five types of moves. It turns out that the homology is still 
a well-defined invariant but we lose the absolute bi-grading. 
\begin{THE}
Given two diagrams $D$ and $D'$ of an oriented spatial web $\Gamma$, 
the two chain complexes $F(D)$ and $F(D')$ are chain homotopy equivalent up to a shifting of the bi-grading.
In particular, the
 \emph{relatively} bi-graded abelian group
$\mathcal{H}(\Gamma):=H(F(D))$ is a well-defined invariant of the oriented spatial web.
\end{THE}
There is a $R_\Gamma$-module structure on $\mathcal{H}(\Gamma)$ where $R_\Gamma$ is a ring defined in Section \ref{module-str} which
does not depend on the embedding of $\Gamma$.

In \cite{KM-jsharp}, Kronheimer and Mrowka defined two versions of instanton Floer homologies 
$J^\sharp$ and $I^\sharp$ for spatial webs. 
Their definition of webs is more general than the one used in this paper.  
The two instanton Floer homology theories 
$J^\sharp$ and $I^\sharp$ are defined only with $\mathbb{F}$-coefficients where $\mathbb{F}$ is the field of two elements. 
Given a (oriented) spatial web $\Gamma$, $\mathcal{H}(\Gamma;\mathbb{F})$ and $J^\sharp(\Gamma)$ are both equipped with 
$\mathcal{R}_\Gamma$-module structures where $\mathcal{R}_\Gamma:=R_\Gamma\otimes_\mathbb{Z}\mathbb{F}$. 
From \cite{KM-jsharp} it is known that
$\mathcal{H}(\Gamma;\mathbb{F})$ and $J^\sharp(\Gamma)$ are isomorphic for any planar web $\Gamma$. 
More generally, we have the following.
\begin{THE}\label{s-sequence*}
Let $\Gamma$ be an oriented spatial web. There is a spectral sequence of $\mathcal{R}_\Gamma$-modules whose $E_2$-page is the $\mathfrak{sl}(3)$
Khovanov module $\mathcal{H}(\Gamma;\mathbb{F})$ and which converges to $J^\sharp(\Gamma)$.
\end{THE}
This spectral sequence is suggested by Kronheimer \cite{KM-sl3-ss}. 
A similar result in the $\mathfrak{sl}(2)$ Khovanov homology case is proved
in \cite{HN-module}: they showed that  Ozsv\'{a}th-Szab\'{o}'s spectral sequence in \cite{OS-ss} 
respects the module structures on
Khovanov homology and (hat) Heegaard Floer homology. 

Given an oriented spatial web $\Gamma$ with mark points $\boldsymbol{\delta}=\{\delta_i\}$ in the interior of edges, we 
define a relatively bi-graded homological invariant $\mathcal{H}(\Gamma,\boldsymbol{\delta})$ by imitating the definition of pointed 
$\mathfrak{sl}(2)$ Khovanov homology
in \cite{BLS}. When $\boldsymbol{\delta}=\emptyset$, it is nothing but $\mathcal{H}(\Gamma)$. We construct a spectral sequence
relating  $\mathcal{H}(\Gamma,\boldsymbol{\delta};\mathbb{F})$ to $I^\sharp(\Gamma)$.
\begin{THE}\label{I-ss*}
Suppose $\Gamma$ is a \emph{connected} oriented spatial web and ${\boldsymbol{\delta}}=\{\delta_i\}$ is a collection of 
points in the interior of edges of $\Gamma$ such that the homology classes of meridians around $\delta_i$ form a basis
of $H_1(S^3\setminus \Gamma;\mathbb{F})$. Then there is a spectral sequence whose $E_2$-page is 
$\mathcal{H}(\Gamma,{\boldsymbol{\delta}};\mathbb{F})$ and which converges to $I^\sharp(\Gamma)$.
\end{THE}
Let $\Theta$ be the planar theta graph. As an application of the above spectral sequence, we prove the following detection result.
\begin{THE}\label{theta-detection}
Suppose $\Gamma$ is a  spatial theta graph and 
$\boldsymbol{\delta}=\{\delta_1,\delta_2\}$ are two mark points lying on two distinct
edge of $\Gamma$. Then the following are equivalent: 
\begin{enumerate}[label=(\alph*)]
  \item $\Gamma$ is the planar theta graph;
  \item $\mathcal{H}(\Gamma;\mathbb{F})$ and $\mathcal{H}(\Theta;\mathbb{F})$ are isomorphic as  $\mathcal{R}_\Theta$-modules;
  \item $\rank_\mathbb{F} \mathcal{H}(\Gamma,\boldsymbol{\delta};\mathbb{F})=4$;
  \item $J^\sharp(\Gamma)$ and $\mathcal{H}(\Theta;\mathbb{F})$ are isomorphic as $\mathcal{R}_\Theta$-modules;
  \item $\rank_\mathbb{F} I^\sharp(\Gamma)=4$.
\end{enumerate}
\end{THE} 

In this article, all the Khovanov homologies are defined over $\mathbb{Z}$ unless otherwise specified. 
All the instanton Floer homologies are defined over $\mathbb{F}$($=\mathbb{Z}/2$) unless otherwise specified.

\emph{Acknowledgments.} The author learned the spectral sequence in Theorem \ref{s-sequence*} from Peter Kronheimer and
wishes to thank him for his generosity in sharing ideas.

\section{ $\mathfrak{sl}(3)$ Khovanov module for spatial webs}\label{sl3Kh}
In this section, we first review Khovanov's definition of $\mathfrak{sl}(3)$ link homology \cite{Kh-sl3}, which is a 
bi-graded abelian group.
Then we will show that the $\mathfrak{sl}(3)$ homology can be generalized for spatial bipartite trivalent graphs.

\subsection{  $\mathfrak{sl}(3)$ homology for links}
All the contents in this subsection are from \cite{Kh-sl3}.
\begin{DEF}
A (closed) \emph{pre-foam} consists of
\begin{itemize}
  \item  A compact 2-dimensional CW-complex $\Sigma$ such that any point on it has a neighborhood which is homeomorphic to either a 2-dimensional disk or the product of
        letter $Y$ and an interval. The points with neighborhood ``Y''$\times I$ form a collection of circles called \emph{seams}. The complement of
        the seams in $\Sigma$ is a 2-manifold whose connected components are \emph{facets} of the foam. We require the facets are orientable.
  \item For each seam $C$, a cyclic order is chosen for the three facets whose closure includes $C$.
  \item Each facet is decorated with  a number of ``dots'' (possibly empty).
\end{itemize}
A \emph{pre-foam} embedded in $\mathbb{R}^3$ is called a \emph{foam}.
\end{DEF}

It is shown in \cite{Kh-sl3} that
\begin{PR}\label{4axiom}
There is a unique map $F:\{\text{closed pre-foams}\}\to \mathbb{Z}$ characterized by the following four axioms
\begin{enumerate}
  \item If $\Sigma$ is a closed orientable surface decorated with dots, then $F(\Sigma)=0$ unless
      \begin{itemize}
        \item $\Sigma$ is a 2-sphere with two dots, then $F(\Sigma)=-1$.
        \item $\Sigma$ is a torus without any dots, then $F(\Sigma)=3$.
      \end{itemize}
  \item If $\Sigma_1$ and $\Sigma_2$ are pre-foams, then $F(\Sigma_1\sqcup \Sigma_2)=F(\Sigma_1)F(\Sigma_2)$.
  \item Do a surgery along a circle inside a facet of $\Sigma$ to obtain three pre-foams $\Sigma_1, \Sigma_2, \Sigma_3$ as in Figure \ref{Sigma123}. Then
        \begin{equation*}
          -F(\Sigma)=F(\Sigma_1)+F(\Sigma_2)+F(\Sigma_3)
        \end{equation*}
  \item Let $\Theta(k_1,k_2,k_3)$ be the theta foam with $k_i$ dots on the $i$-th facet (see Figure \ref{theta-foam}), then
        \begin{equation*}
          F(\Theta(k_1,k_2,k_3))=\left\{
                                   \begin{array}{ll}
                                     1, & \hbox{if $(k_1,k_2,k_3)=(0,1,2)$, up to cylic permutation;} \\
                                     -1, & \hbox{if $(k_1,k_2,k_3)=(0,2,1)$, up to cylic permutation;} \\
                                     0, & \hbox{otherwise.}
                                   \end{array}
                                 \right.
        \end{equation*}
\end{enumerate}
\end{PR}
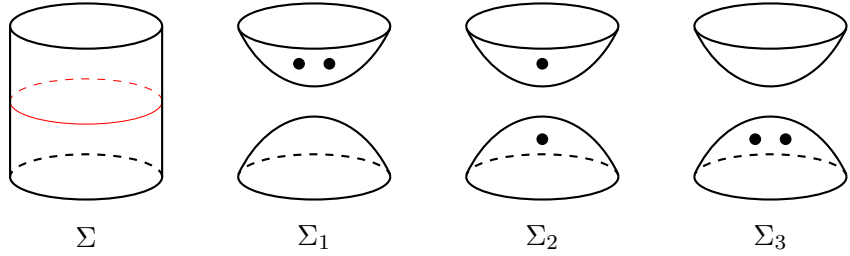
\begin{figure}
\begin{tikzpicture}
\tikzset{
    partial ellipse/.style args={#1:#2:#3}{
        insert path={+ (#1:#3) arc (#1:#2:#3)}
    }
}

\draw[thick]  (0,1) ellipse (1cm and 0.3cm);

\draw[red]  (0,0) [partial ellipse=-180:0: 1cm and 0.3cm];
\draw[dashed,red]  (0,0) [partial ellipse=0:180: 1cm and 0.3cm];

\draw[thick]  (0,-1) [partial ellipse=-180:0: 1cm and 0.3cm];
\draw[thick,dashed]  (0,-1) [partial ellipse=0:180: 1cm and 0.3cm];

\draw[thick] (-1,-1) to (-1,1);   \draw[thick] (1,-1) to (1,1);
\node at (0,-1.8) {$\Sigma$};

\draw[thick]  (3,1) ellipse (1cm and 0.3cm);
\draw[thick]  (3,-1) [partial ellipse=-180:0: 1cm and 0.3cm]; \draw[thick]  (3,-2) [partial ellipse=35:145: 1.22cm and 1.8cm];
\draw[thick,dashed]  (3,-1) [partial ellipse=0:180: 1cm and 0.3cm]; \draw[thick]  (3,2) [partial ellipse=-35:-145: 1.22cm and 1.8cm];
\draw[fill] (3.2,0.5) circle (2pt); \draw[fill] (2.8,0.5) circle (2pt);
\node at (3,-1.8) {$\Sigma_1$};

\draw[thick]  (6,1) ellipse (1cm and 0.3cm);
\draw[thick]  (6,-1) [partial ellipse=-180:0: 1cm and 0.3cm]; \draw[thick]  (6,-2) [partial ellipse=35:145: 1.22cm and 1.8cm];
\draw[thick,dashed]  (6,-1) [partial ellipse=0:180: 1cm and 0.3cm]; \draw[thick]  (6,2) [partial ellipse=-35:-145: 1.22cm and 1.8cm];
\draw[fill] (6,0.5) circle (2pt); \draw[fill] (6,-0.5) circle (2pt);
\node at (6,-1.8) {$\Sigma_2$};

\draw[thick]  (9,1) ellipse (1cm and 0.3cm);
\draw[thick]  (9,-1) [partial ellipse=-180:0: 1cm and 0.3cm]; \draw[thick]  (9,-2) [partial ellipse=35:145: 1.22cm and 1.8cm];
\draw[thick,dashed]  (9,-1) [partial ellipse=0:180: 1cm and 0.3cm]; \draw[thick]  (9,2) [partial ellipse=-35:-145: 1.22cm and 1.8cm];
\draw[fill] (9.2,-0.5) circle (2pt); \draw[fill] (8.8,-0.5) circle (2pt);
\node at (9,-1.8) {$\Sigma_3$};

\end{tikzpicture}
\caption{Do surgery along the red circle in $\Sigma$, we obtain foams $\Sigma_1,\Sigma_2.\Sigma_3$.}\label{Sigma123}
\end{figure}
\begin{figure}
\begin{tikzpicture}
\tikzset{
    partial ellipse/.style args={#1:#2:#3}{
        insert path={+ (#1:#3) arc (#1:#2:#3)}
    }
}

\draw[thick] (0,0) ellipse (2cm and 2.2cm);
\draw[ultra thick] (0,0) [partial ellipse=-180:0: 2cm and 1cm];
\draw[ultra thick,dashed, fill=gray!30,opacity=0.5] (0,0) [partial ellipse=-180:180: 2cm and 1cm];

\node at (-1,0) {1}; \node at (-1,1.5) {2}; \node at (-1,-1.5) {3};
\end{tikzpicture}
\caption{The theta-foam consists of three disk bounded by the same circle. The three disks are marked by 1, 2 and 3.}\label{theta-foam}
\end{figure}
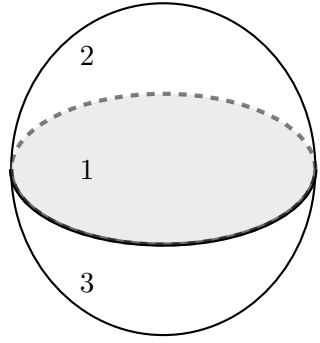

\begin{DEF}
An oriented (closed) foam is a foam  with all the facets oriented such that
 for any seam $C$, the induced orientations from the three nearby facets coincide and cyclic ordering of the three nearby facets are determined by the
 orientation of $C$ and the left-hand-rule.
\end{DEF}

We also want to define oriented foams with boundary. Before that we define the following.
\begin{DEF}\label{web}
A \emph{web} is a bipartite trivalent  
graph (possibly including loops without vertices). An oriented web is a directed bipartite trivalent  graph such that
at each vertex all the edges are either all incoming or all outgoing.
\end{DEF}

\begin{DEF}
An oriented foam with boundary is the intersection of $\mathbb{R}^2\times [0,1]$ and a closed foam $\Sigma$ such that
$\mathbb{R}^2\times\{0\}$ and $\mathbb{R}^2\times\{1\}$ are transversal to the facets and seams of $\Sigma$. In particular,
$\Gamma_1=\mathbb{R}^2\times\{0\}\cap \Sigma$ and $\Gamma_2=\mathbb{R}^2\times\{1\}\cap \Sigma$ are two oriented planar webs. The oriented foam with boundary
can be thought as a cobordism from $\Gamma_1$ to $\Gamma_2$.
\end{DEF}
By composing with the forgetful map from oriented closed foams to pre-foams, we can define
$F:\{\text{oriented closed foams}\}\to \mathbb{Z}$. Here we still denote it by $F$ by abuse of notation.  Next we want to extend $F$ into
a (1+1)-dimensional TQFT on oriented planar webs and foams. Given two oriented planar webs $\Gamma_1$ and $\Gamma_2$, let $\Homm_{\text{OF}} (\Gamma_1,\Gamma_2)$
be the set of all cobordisms (oriented foams) from $\Gamma_1$ to $\Gamma_2$. Define
\begin{equation*}
  F(\Gamma):=\mathbb{Z}\Homm_{\text{OF}} (\emptyset,\Gamma)/\{\sum_i a_i\Sigma_i|
   \sum_i a_i F(\Phi \circ \Sigma_i)=0~\text{for all}~\Phi \in \Homm_{\text{OF}} (\Gamma,\emptyset)  \}
\end{equation*}
for an oriented planar web $\Gamma$. $F(\Gamma)$ can be equipped with a $\mathbb{Z}$-grading by defining
\begin{equation*}
  \deg (\Sigma):= \chi (\partial \Sigma)-2\chi(\Sigma\setminus \text{Nbh}(\text{dots}))
\end{equation*}
for $\Sigma\in \Homm_{\text{OF}} (\emptyset,\Gamma)$ (see \cite{Kh-sl3}*{Section 3.3} for more details).
It is clear from the definition that given $\Phi\in \Homm_{\text{OF}} (\Gamma_1,\Gamma_2)$, there is a well-defined homomorphism
$F(\Phi):F(\Gamma_1)\to F(\Gamma_2)$ of degree $\deg \Phi$.

Let $L$ be an oriented link in $\mathbb{R}^3$ and $D$ be a diagram of $L$ with $n$ crossings. A crossing is called a positive or negative crossing
according to Figures \ref{+01} and \ref{-01}. \footnote{The convention used here is non-standard. We use this non-standard convention in order to be consistent with \cite{Kh-sl3}.}
\begin{figure}
\begin{tikzpicture}
\tikzset{middlearrow/.style={
        decoration={markings,
            mark= at position 0.5 with {\arrow{#1}} ,
        },
        postaction={decorate}
    }
}

\draw[thick,-latex] (1,-1) to (-1,1); \draw[thick,-latex,dash pattern=on 1.3cm off 0.25cm] (-1,-1) to (1,1);  \node[below] at (0,-1.2) {Positive crossing};

\draw[thick,-latex] (2,-1)  to [out=45,in=270]  (2.7,0) to [out=90,in=315]    (2,1);  \node[below] at (3,-1.2) {0-resolution};
\draw[thick,-latex] (4,-1)  to [out=135,in=270]  (3.3,0) to [out=90,in=225]   (4,1);

\draw[thick,-latex] (5,-1) to (6,-0.4);  \draw[thick,middlearrow={latex}] (6,0.4) to (6,-0.4);
\draw[thick,-latex] (7,-1) to (6,-0.4);
\draw[thick,-latex] (6,0.4) to (5,1);   \draw[thick,-latex] (6,0.4) to (7,1);
\node[below] at (6,-1.2) {1-resolution};
\end{tikzpicture}
\captionof{figure}{Positive crossing and two types of resolutions}\label{+01}
\vspace{0.5cm}

\begin{tikzpicture}
\tikzset{middlearrow/.style={
        decoration={markings,
            mark= at position 0.5 with {\arrow{#1}} ,
        },
        postaction={decorate}
    }
}

\draw[thick,-latex, dash pattern=on 1.3cm off 0.25cm] (1,-1) to (-1,1); \draw[thick,-latex] (-1,-1) to (1,1);  \node[below] at (0,-1.2) {Negative crossing};

\draw[thick,-latex] (2,-1) to (3,-0.4);  \draw[thick,middlearrow={latex}] (3,0.4) to (3,-0.4);
\draw[thick,-latex] (4,-1) to (3,-0.4);
\draw[thick,-latex] (3,0.4) to (2,1);   \draw[thick,-latex] (3,0.4) to (4,1);
\node[below] at (3,-1.2) {0-resolution};

\draw[thick,-latex] (5,-1)  to [out=45,in=270]  (5.7,0) to [out=90,in=315]    (5,1);  \node[below] at (6,-1.2) {1-resolution};
\draw[thick,-latex] (7,-1)  to [out=135,in=270]  (6.3,0) to [out=90,in=225]   (7,1);
\end{tikzpicture}
\captionof{figure}{Negative crossing and two types of resolutions}\label{-01}
\end{figure}
 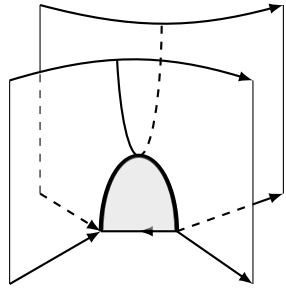
\begin{figure}
\begin{tikzpicture}
\tikzset{middlearrow/.style={
        decoration={markings,
            mark= at position 0.5 with {\arrow{#1}} ,
        },
        postaction={decorate}
    }
}
\tikzset{
    partial ellipse/.style args={#1:#2:#3}{
        insert path={+ (#1:#3) arc (#1:#2:#3)}
    }
}

\draw[thick,middlearrow={latex}] (1,0) to (0,0);
\draw[thick,-latex] (-1.2,-0.7) to (0,0);
\draw[thick, dashed,-latex] (-0.8,0.5) to (0,0); \draw[dashed] (-0.8,0.5) to (-0.8,2.1); \draw (-0.8,2.1) to (-0.8, 3);
\draw (-1.2,-0.7) to (-1.2,2);
\draw[thick,-latex] (1,0) to (2,-0.7);
\draw (2,-0.7) to (2,2);   \draw[thick,-latex] (-1.2,2) to [out=17,in=163] (2,2);
\draw[thick, dashed] (1,0) to  (2,2.5/7); \draw[thick,-latex] (2,2.5/7) to    (2.4,0.5);
\draw (2.4,0.5) to (2.4,3); \draw[thick,-latex] (-0.8, 3) to [out=-15, in=195]   (2.4,3);

\draw[ultra thick] (0.5,0) [partial ellipse=0:180: 0.5cm and 1cm];
\draw[fill=gray!30, opacity=0.5] (0.5,0.02) [partial ellipse=0:180: 0.47cm and 0.97cm];

\draw[thick] (0.5,2.8) [partial ellipse=197:270: 0.3cm and 1.8cm];
\draw[thick, dashed] (0.5,2.8) [partial ellipse=270:358: 0.3cm and 1.8cm];
\end{tikzpicture}
\caption{The skein cobordism between two resolutions. The bold arc is the seam.}\label{skein-co}
\end{figure}

For each crossing, there are two ways to resolve the crossing: 0-resolution and 1-resolution, see Figures \ref{+01} and \ref{-01}. Given any $v\in \{0,1\}^n$, we can resolve all the
crossings by 0- or 1-resolution determined by $v$ and obtain an oriented planar web $D_v$.
If $v,u\in \{0,1\}^n$ only differ at one crossing ($v$ assigns $0$ and $u$ assigns $1$ to this crossing), then there is a cobordism
$S_{vu}$ from $D_v$ to $D_u$ which
is a product cobordism away from the crossing and is the cobordism in Figure \ref{skein-co} near the crossing.

Let $p_+$ and $p_-$ be the numbers of positive and negative crossings in $D$ respectively. We use $F(\Gamma)\{l\}$ to denote the graded abelian group
obtained by increasing the (quantum) grading of $F(\Gamma)$ by $l$.
For each vertex $v$ of the cube $\{0,1\}^n$, we assign a graded abelian group $F(D_v)\{3p_- - 2p_+ -|v|_1\}$ where $|v|_1$ is the number of 1's in $v$.
For each edge $vu$ of the cube $\{0,1\}^n$, we assign
the map $F(S_{vu}):F(D_v)\{3p_- - 2p_+ -|v|_1\}\to F(D_u)\{3p_- - 2p_+ -|u|_1\}$. This map preserves the grading. After adding plus or minus sign to those maps
associated with the edges appropriately, each square in the cube anti-commutes. The total complex $F(D)$ of the cube becomes a chain complex of graded abelian groups
whose homological degree $i$ term is
\begin{equation}\label{cube}
  F(D)_i:=\bigoplus_{|v|_1=i+p_-} F(D_v)\{3p_- - 2p_+ -|v|_1\}
\end{equation}
Alternatively, we can define $F(D)$ inductively by Figure \ref{ind-cube}.
The homology of $F(D)$ is a bi-graded abelian group: one grading is the \emph{homological grading} and
the other grading (we call it the \emph{quantum grading}) is obtained from the grading of $F(D_v)\{3p_- - 2p_+ -|v|_1\}$.

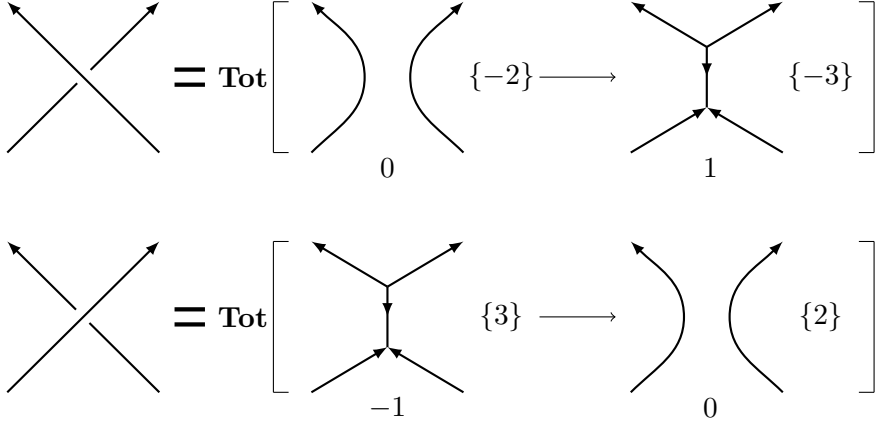
\begin{figure}
\begin{minipage}{1\textwidth}
\begin{tikzpicture}
\tikzset{middlearrow/.style={
        decoration={markings,
            mark= at position 0.5 with {\arrow{#1}} ,
        },
        postaction={decorate}
    }
}

\draw[thick,-latex] (1,-1) to (-1,1); \draw[thick,-latex,dash pattern=on 1.3cm off 0.25cm] (-1,-1) to (1,1);
\draw[ultra thick]  (1.2,0.1) to (1.6,0.1); \draw[ultra thick]  (1.2,-0.1) to (1.6,-0.1);

\node at (2.1,0) {$\textbf{Tot}$};

\draw (2.5,-1) to (2.5,1); \draw (2.5,1) to (2.7,1);\draw (2.5,-1) to (2.7,-1);

\draw[thick,-latex] (3,-1)  to [out=45,in=270]  (3.7,0) to [out=90,in=315]    (3,1);
\draw[thick,-latex] (5,-1)  to [out=135,in=270]  (4.3,0) to [out=90,in=225]   (5,1);
\node at (5.5,0) {$\{-2\}$};  \node at (4,-1.2) {0}; \node at (8.25,-1.2) {1};

\draw[->] (6,0) to (7,0);

\draw[thick,-latex] (7.2,-1) to (8.2,-0.4);  \draw[thick,middlearrow={latex}] (8.2,0.4) to (8.2,-0.4);
\draw[thick,-latex] (9.2,-1) to (8.2,-0.4);
\draw[thick,-latex] (8.2,0.4) to (7.2,1);   \draw[thick,-latex] (8.2,0.4) to (9.2,1);
\node at (9.7,0) {$\{-3\}$ };

\draw (10.4,-1) to (10.4,1); \draw (10.4,1) to (10.2,1);  \draw (10.4,-1) to (10.2,-1);
\node at (0,-2) {};
\end{tikzpicture}
\end{minipage}

\begin{minipage}{1\textwidth}
\begin{tikzpicture}
\tikzset{middlearrow/.style={
        decoration={markings,
            mark= at position 0.5 with {\arrow{#1}} ,
        },
        postaction={decorate}
    }
}

\draw[thick,-latex, dash pattern=on 1.3cm off 0.25cm] (1,-1) to (-1,1); \draw[thick,-latex] (-1,-1) to (1,1);
\draw[ultra thick]  (1.2,0.1) to (1.6,0.1); \draw[ultra thick]  (1.2,-0.1) to (1.6,-0.1);

\node at (2.1,0) {$\textbf{Tot}$};

\draw (2.5,-1) to (2.5,1); \draw (2.5,1) to (2.7,1);\draw (2.5,-1) to (2.7,-1);

\draw[thick,-latex] (3,-1) to (4,-0.4);  \draw[thick,middlearrow={latex}] (4,0.4) to (4,-0.4);
\draw[thick,-latex] (5,-1) to (4,-0.4);
\draw[thick,-latex] (4,0.4) to (3,1);   \draw[thick,-latex] (4,0.4) to (5,1);

\node at (5.5,0) {$\{3\}$};  \node at (4,-1.2) {$-1$}; \node at (8.25,-1.2) {0};

\draw[->] (6,0) to (7,0);

\draw[thick,-latex] (7.2,-1)  to [out=45,in=270]  (7.9,0) to [out=90,in=315]    (7.2,1);
\draw[thick,-latex] (9.2,-1)  to [out=135,in=270]  (8.5,0) to [out=90,in=225]   (9.2,1);

\node at (9.7,0) {$\{2\}$ };

\draw (10.4,-1) to (10.4,1); \draw (10.4,1) to (10.2,1);  \draw (10.4,-1) to (10.2,-1);

\end{tikzpicture}
\end{minipage}
\caption{\textbf{Tot} denote the total complexes of the double complexes on the right. $-1,0,1$ under the webs
 means the horizontal (homological) degree of the double complexes.}\label{ind-cube}
\end{figure}

The diagram $D$ of the link $L$ is not unique. But any two diagrams of $L$ can be connected by a sequence of three types 
of Reidemeister moves.
\begin{THE}[Khovanov]
Given two diagrams $D$ and $D'$ of the link $L$, the two chain complexes $F(D)$ and $F(D')$ are chain homotopy equivalent. In particular, the
 bi-graded abelian group
$\mathcal{H}(L):=H(F(D))$ is a well-defined invariant of $L$.
\end{THE}

\subsection{Spatial webs and Reidemeister moves}
A \emph{spatial web} is a web embedded in $\mathbb{R}^3$. In this subsection, we want to generalized Khovanov's link homology $\mathcal{H}(L)$ to
oriented spatial webs. Similar to links, we can also choose a diagram for a spatial web and require that the crossings are disjoint from
 the vertices.
The diagram for a spatial web is not unique. Any two diagrams for the same spatial web are connected by a sequence of \emph{five} types of Reidemeister moves
(see \cite{Kau-graph} or \cite{Kau-graph2} for more details).
In addition to the three types of Reidemeister moves in the link case, there are two new Reidemeister moves shown in 
Figures \ref{RIV} and \ref{RV}.
\begin{figure}
\scalebox{0.85}{
\begin{minipage}{0.6\textwidth}
\begin{tikzpicture}
\draw[thick] (-1,-1) to (0,0); \draw[thick] (1,-1) to (0,0); \draw[thick] (0,0)to (0,1);
\draw[thick, dash pattern=on 0.62cm off 0.15cm on 1.07cm off 0.15cm on 100cm] (-1.3,-0.6) to (1.3,-0.6);

\draw[thick,<->] (1.3,0) to (2.3,0);

\draw[thick] (2.5,-1) to (3.5,0); \draw[thick] (4.5,-1) to (3.5,0); \draw[thick] (3.5,0)to (3.5,1);
\draw[thick, dash pattern=on 1.24cm off 0.15cm  on 100cm] (2.2,0.6) to (4.8,0.6);

\node at (-2,0) {RIV:};
\end{tikzpicture}
\end{minipage}%
\begin{minipage}{0.4\textwidth}
\begin{tikzpicture}
\draw[thick, dash pattern=on 0.50cm off 0.15cm on 100cm] (-1,-1) to (0,0);
\draw[thick, dash pattern=on 0.50cm off 0.15cm on 100cm] (1,-1) to (0,0); \draw[thick] (0,0)to (0,1);
\draw[thick] (-1.3,-0.6) to (1.3,-0.6);

\draw[thick,<->] (1.3,0) to (2.3,0);

\draw[thick] (2.5,-1) to (3.5,0); \draw[thick] (4.5,-1) to (3.5,0); \draw[thick, dash pattern=on 0.52cm off 0.15cm on 100cm] (3.5,0)to (3.5,1);
\draw[thick] (2.2,0.6) to (4.8,0.6);

\end{tikzpicture}

\end{minipage}
}
\caption{Reidemeister moves of type IV}\label{RIV}
\end{figure}
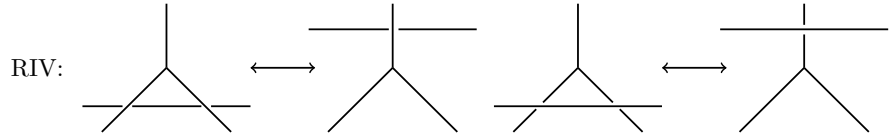
\begin{figure}

\begin{tikzpicture}
\draw[thick] (0,1) to (0,2);
\draw[thick,dash pattern=on 1.04cm off 0.15cm on 100cm] (1,-1) to [out=150, in=-90]  (-0.5,0.2); \draw[thick] (-0.5,0.2) to [out=90,in=210]  (0,1);
\draw[thick] (-1,-1) to [out=30, in=-90]  (0.5,0.2); \draw[thick] (0.5,0.2) to [out=90,in=-30]  (0,1);

\draw[thick,<->] (1,0.2) to (2,0.2);

\draw[thick] (3,1) to (3,2);
\draw[thick] (2,-1) to  [out=30, in=-90] (3,1);
\draw[thick] (4,-1) to  [out=150, in=-90] (3,1);

\draw[thick,<->] (4,0.2) to (5,0.2);

\draw[thick] (6,1) to (6,2);
\draw[thick] (7,-1) to [out=150, in=-90]  (5.5,0.2); \draw[thick] (5.5,0.2) to [out=90,in=210]  (6,1);
\draw[thick,,dash pattern=on 1.04cm off 0.15cm on 100cm] (5,-1) to [out=30, in=-90]  (6.5,0.2); \draw[thick] (6.5,0.2) to [out=90,in=-30]  (6,1);

\node at  (-1.5,0.2) {RV:};
\end{tikzpicture}
\caption{Reidemeister moves of type V}\label{RV}
\end{figure}
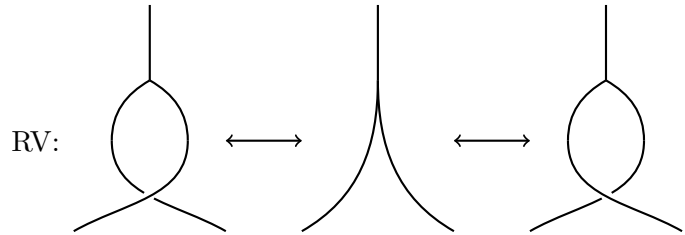

Suppose $\Gamma$ is an oriented spatial web and $D$ is a diagram for it with $n$ crossings. The we can form a chain complex $F(D)$ 
as in \eqref{cube}.
The same argument as in the link case shows that the chain homotopy equivalence class of $F(D)$ is invariant under the first three types of Reidemeister moves.
In order to obtain a well-defined invariant for $\Gamma$, we need to study the change of $F(D)$ under moves of types IV and V. We deal with RV first.
Pick an orientation for a type V move as in Figure \ref{RVDD'},
 we want to compare the chain complexes $F(D)$ and $F(D')$.
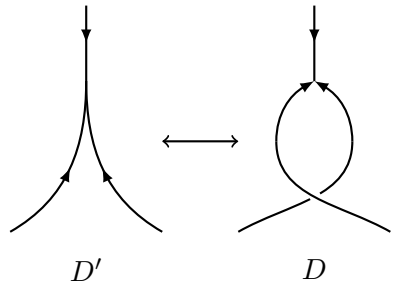
\begin{figure}
\begin{tikzpicture}
\tikzset{middlearrow/.style={
        decoration={markings,
            mark= at position 0.5 with {\arrow{#1}} ,
        },
        postaction={decorate}
    }
}
\draw[thick,middlearrow={latex}] (3,2) to (3,1);
\draw[thick,middlearrow={latex}] (2,-1) to  [out=30, in=-90] (3,1);
\draw[thick,middlearrow={latex}] (4,-1) to  [out=150, in=-90] (3,1);

\node at (3,-1.5) {$D'$};

\draw[thick,<->] (4,0.2) to (5,0.2);

\draw[thick,middlearrow={latex}] (6,2) to (6,1);
\draw[thick,] (7,-1) to [out=150, in=-90]  (5.5,0.2); \draw[thick,-latex] (5.5,0.2) to [out=90,in=210]  (6,1);
\draw[thick,dash pattern=on 1.04cm off 0.15cm on 100cm] (5,-1) to [out=30, in=-90]  (6.5,0.2); \draw[thick,-latex] (6.5,0.2) to [out=90,in=-30]  (6,1);

\node at (6,-1.5) {$D$};
\end{tikzpicture}
\caption{A Reidemeister move of type V for an oriented spatial web.}\label{RVDD'}
\end{figure}

\begin{PR}\label{RVD=D'}
Let $D$ and $D'$ be the two diagrams in Figure \ref{RVDD'}.
Up to a shifting of the homological and quantum gradings, $F(D)$ and $F(D')$ are chain homotopy equivalent.
\end{PR}
Before moving to the proof of this proposition, we need to state a result from \cite{Kh-sl3}.
\begin{LE}\cite{Kh-sl3}*{Proposition 8}\label{digon-homology}
Suppose $\Gamma_1$ and $\Gamma_2$ are two oriented planar webs  such that $\Gamma_2$ is obtained by removing a digon from $\Gamma_1$ as shown in
 in Figure \ref{digon-webs}. Then
 $$g=F(\Sigma_{d1})\oplus F(\Sigma_{d2}):F(\Gamma_1)\to F(\Gamma_2)\{1\}\oplus F(\Gamma_2)\{-1\}$$
  is an isomorphism
 where $\Sigma_{d1}$ and $\Sigma_{d2}$ are the cobordisms in Figure \ref{digon-co} (thought as cobordisms from the top to the bottom). Reversing the direction
 of the two cobordisms (i.e. regard them as cobordisms from the bottom to the top), we obtain $\overline{\Sigma}_{d1}$ and $\overline{\Sigma}_{d2}$. The compositions 
 $g\circ F(\overline{\Sigma}_{d1})$ and
 $g\circ F(\overline{\Sigma}_{d2})$ are embeddings onto the
 $F(\Gamma_2)\{-1\}$ and $F(\Gamma_2)\{1\}$ summands respectively (be careful about the order change).
\end{LE}
\begin{figure}
\begin{tikzpicture}
\tikzset{middlearrow/.style={
        decoration={markings,
            mark= at position 0.5 with {\arrow{#1}} ,
        },
        postaction={decorate}
    }
}
\tikzset{
    partial ellipse/.style args={#1:#2:#3}{
        insert path={+ (#1:#3) arc (#1:#2:#3)}
    }
}

\node at (-1,0.5) {$\Gamma_1:$};
\draw[thick,-latex] (0,2) to (0,1); \draw[thick,-latex] (0,0) to (0,-1);

\draw[thick,middlearrow={latex}] (0,0.5) [partial ellipse=-90:90: 0.3cm and 0.5cm];
\draw[thick,middlearrow={latex}] (0,0.5) [partial ellipse=270:90: 0.3cm and 0.5cm];

\node at (2,0.5) {$\Gamma_2:$};
\draw[thick, -latex] (3,2) to (3,-1);
\end{tikzpicture}
\caption{}\label{digon-webs}
\end{figure}

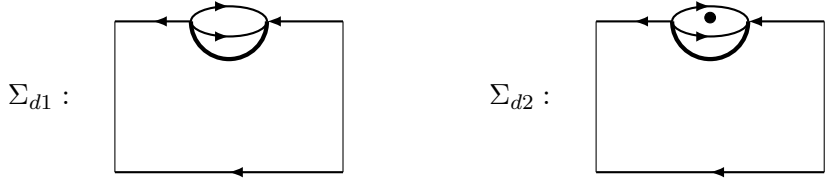
\begin{figure}
\centering
\begin{minipage}{0.5\textwidth}
\begin{tikzpicture}
\tikzset{middlearrow/.style={
        decoration={markings,
            mark= at position 0.5 with {\arrow{#1}} ,
        },
        postaction={decorate}
    }
}
\tikzset{
    partial ellipse/.style args={#1:#2:#3}{
        insert path={+ (#1:#3) arc (#1:#2:#3)}
    }
}

\draw[thick,middlearrow={latex}] (1,0) to (0,0);
\draw[thick,middlearrow={latex}]  (1.5,0) [partial ellipse=180:0: 0.5cm and 0.2cm];
\draw[thick,middlearrow={latex}]  (1.5,0) [partial ellipse=180:360: 0.5cm and 0.2cm];
\draw[thick,-latex] (3,0) to (2,0);

\draw[ultra thick]  (1.5,0) [partial ellipse=180:360: 0.5cm and 0.5cm];

\draw (3,0) to (3,-2);
\draw (0,0) to (0,-2);

\draw[thick,middlearrow={latex}] (3,-2) to (0,-2);

\node at (-1,-1) {$\Sigma_{d1}:$};
\end{tikzpicture}
\end{minipage}%
\begin{minipage}{0.5\textwidth}
\begin{tikzpicture}
\tikzset{middlearrow/.style={
        decoration={markings,
            mark= at position 0.5 with {\arrow{#1}} ,
        },
        postaction={decorate}
    }
}
\tikzset{
    partial ellipse/.style args={#1:#2:#3}{
        insert path={+ (#1:#3) arc (#1:#2:#3)}
    }
}

\draw[thick,middlearrow={latex}] (1,0) to (0,0);
\draw[thick,middlearrow={latex}]  (1.5,0) [partial ellipse=180:0: 0.5cm and 0.2cm];
\draw[thick,middlearrow={latex}]  (1.5,0) [partial ellipse=180:360: 0.5cm and 0.2cm];
\draw[thick,-latex] (3,0) to (2,0);

\draw[ultra thick]  (1.5,0) [partial ellipse=180:360: 0.5cm and 0.5cm];

\draw (3,0) to (3,-2);
\draw (0,0) to (0,-2);

\draw[thick,middlearrow={latex}] (3,-2) to (0,-2);

\draw[fill]  (1.5,0.05) circle (2pt);
\node at (-1,-1) {$\Sigma_{d2}:$};
\end{tikzpicture}
\end{minipage}
\caption{The digon cobordisms.}\label{digon-co}
\end{figure}

\begin{figure}
\begin{tikzpicture}
\tikzset{middlearrow/.style={
        decoration={markings,
            mark= at position 0.5 with {\arrow{#1}} ,
        },
        postaction={decorate}
    }
}
\tikzset{
    partial ellipse/.style args={#1:#2:#3}{
        insert path={+ (#1:#3) arc (#1:#2:#3)}
    }
}

\node at (-1,0) {$D_1:$};
\draw[thick,-latex] (0,2) to (0,1);
\draw[thick,middlearrow={latex}] (0,0.5) [partial ellipse=-90:90: 0.3cm and 0.5cm];
\draw[thick,middlearrow={latex}] (0,0.5) [partial ellipse=270:90: 0.3cm and 0.5cm];
\draw[thick, -latex] (0,0) to (0,-1);
\draw[thick, -latex] (-0.5,-1.5) to (0,-1); \draw[thick, -latex] (0.5,-1.5) to (0,-1);
\end{tikzpicture}
\caption{}\label{RVD1}
\end{figure}

\begin{proof}[Proof of Proposition \ref{RVD=D'}]
Using the inductive definition in Figure \ref{ind-cube}, we have
\begin{equation*}
  F(D)= \textbf{Tot}[F(D')\{-2\}\to F(D_1)\{-3\}]
\end{equation*}
where $D_1$ is shown in Figure \ref{RVD1} and the map from $F(D')$ to $F(D_1)$ is induced by the cobordism in Figure \ref{RVD'D1-co}.
\begin{figure}
\begin{tikzpicture}
\tikzset{middlearrow/.style={
        decoration={markings,
            mark= at position 0.5 with {\arrow{#1}} ,
        },
        postaction={decorate}
    }
}
\tikzset{
    partial ellipse/.style args={#1:#2:#3}{
        insert path={+ (#1:#3) arc (#1:#2:#3)}
    }
}

\draw[thick,-latex] (-1.2,-0.7) to (0,0);
\draw[thick, -latex] (-0.8,0.5) to (0,0);

\draw[thick,middlearrow={latex}] (1,0) to (0,0);
\draw[thick,middlearrow={latex}]  (1.5,0) [partial ellipse=180:0: 0.5cm and 0.2cm];
\draw[thick,middlearrow={latex}]  (1.5,0) [partial ellipse=180:360: 0.5cm and 0.2cm];
\draw[thick,-latex] (3,0) to (2,0);

\draw[ultra thick] (0.5,0) [partial ellipse=0:180: 0.5cm and 1cm];
\draw[fill=gray!30, opacity=0.5] (0.5,0.02) [partial ellipse=0:180: 0.47cm and 0.97cm];

\draw (-1.2,-0.7) to (-1.2,1.3);
\draw[dashed] (-0.8,0.5)  to (-0.8,1.6);      \draw (-0.8,1.6) to (-0.8,2.5);

\draw[thick,middlearrow={latex}] (2,1.3) [partial ellipse=180:90: 3.2cm and 0.6cm];
\draw[thick,middlearrow={latex}] (2,2.5) [partial ellipse=180:270: 2.8cm and 0.6cm];
\draw[thick,-latex] (3,1.9) to (2,1.9);
\draw[ultra thick] (2,1.9) to (2,0);

\draw (3,0) to (3,1.9);

\draw[thick] (0.5,2.8) [partial ellipse=213:270: 0.3cm and 1.8cm];
\draw[ dash pattern=on 0.05cm] (0.5,2.8) [partial ellipse=270:332: 0.3cm and 1.8cm];
\end{tikzpicture}
\caption{The cobordism from $D'$ to $D_1$}\label{RVD'D1-co}
\end{figure}
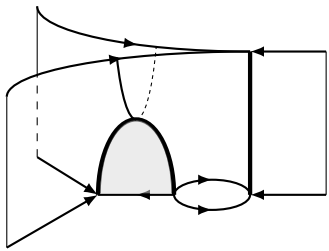
By Lemma \ref{digon-homology}, we have 
$$
F(D_1)\{-3\}=F(D')\{-2\}\oplus F(D')\{-4\}
$$ 
The composition map
\begin{equation*}
 f: F(D')\{-2\}\to F(D_1)\{-3\}\to F(D')\{-2\}
\end{equation*}
is induced by the composition of the digon cobordism in Figure \ref{digon-co} and the cobordism in Figure \ref{RVD'D1-co}. This composition cobordism can be deformed into
the product cobordism (see Figure \ref{compo-digon-D'D1}). Therefore $f$ is the identify map. This implies
\begin{equation*}
  [F(D')\{-2\}\to F(D_1)\{-3\}]\cong [F(D')\{-2\}\to F(D')\{-2\}\oplus F(D')\{-4\}]
\end{equation*}
is chain homotopy equivalent to
\begin{equation*}
  0\to  F(D')\{-4\}
\end{equation*}
we have
\begin{equation*}
  \textbf{Tot}[0\to  F(D')\{-4\}]=F(D')\{-4\}[1]
\end{equation*}
where $[1]$ means increasing the homological grading by $1$. This completes the proof.
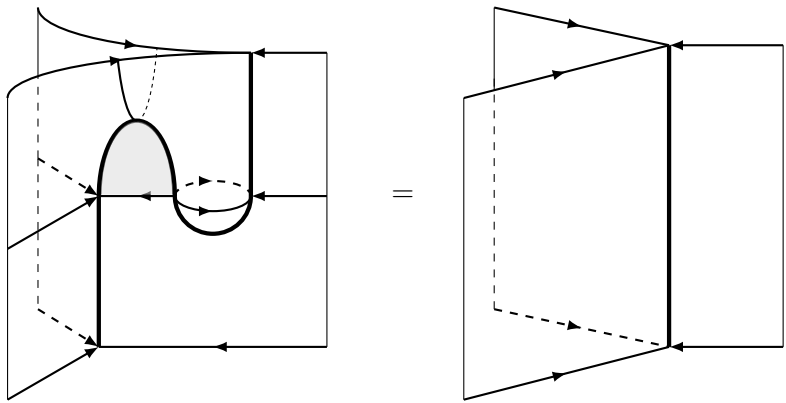
\begin{figure}
\begin{tikzpicture}
\tikzset{middlearrow/.style={
        decoration={markings,
            mark= at position 0.5 with {\arrow{#1}} ,
        },
        postaction={decorate}
    }
}
\tikzset{
    partial ellipse/.style args={#1:#2:#3}{
        insert path={+ (#1:#3) arc (#1:#2:#3)}
    }
}

\draw[thick,middlearrow={latex}] (1,0) to (0,0);
\draw[thick,dashed, middlearrow={latex}]  (1.5,0) [partial ellipse=180:0: 0.5cm and 0.2cm];
\draw[thick,middlearrow={latex}]  (1.5,0) [partial ellipse=180:360: 0.5cm and 0.2cm];
\draw[thick,-latex] (3,0) to (2,0);

\draw[ultra thick]  (1.5,0) [partial ellipse=180:360: 0.5cm and 0.5cm];

\draw[thick,-latex] (-1.2,-0.7) to (0,0);
\draw[thick,dashed, -latex] (-0.8,0.5) to (0,0);

\draw[dashed] (-0.8,0.5) to (-0.8, -0.48);   \draw[dashed] (-0.8, -0.48) to (-0.8,-1.5);

\draw (-1.2,-0.7) to (-1.2,-2.7);
\draw (3,0) to (3,-2);
\draw[ultra thick] (0,0) to (0,-2);

\draw[thick,middlearrow={latex}] (3,-2) to (0,-2);
\draw[thick,-latex] (-1.2,-2.7) to (0,-2);
\draw[thick, dashed, -latex] (-0.8,-1.5) to (0,-2);

\draw[ultra thick] (0.5,0) [partial ellipse=0:180: 0.5cm and 1cm];
\draw[fill=gray!30, opacity=0.5] (0.5,0.02) [partial ellipse=0:180: 0.47cm and 0.97cm];

\draw (-1.2,-0.7) to (-1.2,1.3);
\draw[dashed] (-0.8,0.5)  to (-0.8,1.6);      \draw (-0.8,1.6) to (-0.8,2.5);

\draw[thick,middlearrow={latex}] (2,1.3) [partial ellipse=180:90: 3.2cm and 0.6cm];
\draw[thick,middlearrow={latex}] (2,2.5) [partial ellipse=180:270: 2.8cm and 0.6cm];
\draw[thick,-latex] (3,1.9) to (2,1.9);
\draw[ultra thick] (2,1.9) to (2,0);

\draw (3,0) to (3,1.9);

\draw[thick] (0.5,2.8) [partial ellipse=213:270: 0.3cm and 1.8cm];
\draw[ dash pattern=on 0.05cm] (0.5,2.8) [partial ellipse=270:332: 0.3cm and 1.8cm];

\draw[thick,middlearrow={latex}] (4.8,1.3) to (7.5,2);
\draw[thick,middlearrow={latex}] (5.2,2.5) to (7.5,2);
\draw[thick,-latex] (9,2) to (7.5,2);

\draw (4.8,-2.7) to (4.8,1.3);
\draw[dashed] (5.2,-1.5)  to (5.2,1.6);      \draw (5.2,1.4) to (5.2,2.5);
\draw (9,2) to (9,-2);

\draw[thick,middlearrow={latex}] (4.8,-2.7) to (7.5,-2);
\draw[thick,dashed, middlearrow={latex}] (5.2,-1.5) to (7.5,-2);
\draw[thick,-latex] (9,-2) to (7.5,-2);

\draw[ultra thick] (7.5,2) to (7.5,-2);

\node at (4,0) {=};
\end{tikzpicture}
\caption{The composition of the two cobordisms in Figure \ref{digon-co} and Figure \ref{RVD'D1-co}. 
The bold arcs represent the seams.}\label{compo-digon-D'D1}
\end{figure}
\end{proof}
Proposition \ref{RVD=D'} only deals with one choice of oriented type V move. Similar proofs work for other possible oriented type V moves, so we skip the details for other
situations.

Next we deal with the type IV moves. Again we only deal with a specific oriented type IV move in detail and skip the proofs for other situations since the proofs are similar.

This time we need a ``square lemma'' besides the ``digon lemma'' used before.
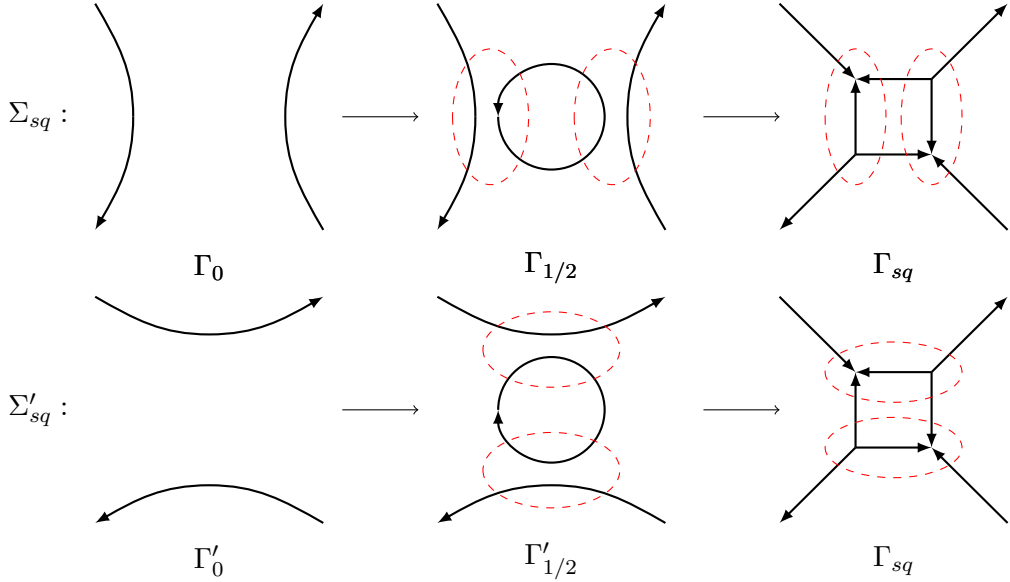
\begin{figure}
\begin{tikzpicture}
\tikzset{middlearrow/.style={
        decoration={markings,
            mark= at position 0.8 with {\arrow{#1}} ,
        },
        postaction={decorate}
    }
}

\tikzset{
    partial ellipse/.style args={#1:#2:#3}{
        insert path={+ (#1:#3) arc (#1:#2:#3)}
    }
}


\draw[thick,-latex] (0,0) to (0,1);
\draw[thick,-latex] (0,0) to (1,0);
\draw[thick,-latex] (1,1) to (0,1);
\draw[thick,-latex] (1,1) to (1,0);

\draw[thick,-latex] (-1,2) to (0,1);
\draw[thick,-latex] (1,1) to (2,2);
\draw[thick,-latex] (0,0) to (-1,-1);
\draw[thick,-latex] (2,-1) to (1,0);

\draw[thick] (-5.5,2) to [out=-60,in=90] (-5,0.5); \draw[thick,-latex] (-5,0.5) to [out=-90,in=60]  (-5.5,-1);
\draw[thick,-latex] ((-2.5,-1) to [out=120,in=-90]  (-3,0.5) to [out=90,in=240] (-2.5,2);

\draw[thick,-latex] (-4,0.5) [partial ellipse=-180:180: 0.7cm and 0.7cm];

\draw[thick] (-10,2) to [out=-60,in=90] (-9.5,0.5); \draw[thick,-latex] (-9.5,0.5) to [out=-90,in=60]  (-10,-1);
\draw[thick,-latex] ((-7,-1) to [out=120,in=-90]  (-7.5,0.5) to [out=90,in=240] (-7,2);

\draw[red,dashed]  (0,0.5) [partial ellipse=-180:180: 0.4cm and 0.9cm];
\draw[red,dashed]  (1,0.5) [partial ellipse=-180:180: 0.4cm and 0.9cm];

\draw[red,dashed]  (-4.8,0.5) [partial ellipse=-180:180: 0.5cm and 0.9cm];
\draw[red,dashed]  (-3.2,0.5) [partial ellipse=-180:180: 0.5cm and 0.9cm];

\draw[->] (-2,0.5) to (-1,0.5);
\draw[->] (-2-5+0.25,0.5) to (-1-5+0.25,0.5);

\node at (0.5,-1.5) {$\Gamma_{sq}$};
\node at (-4,-1.5) {$\Gamma_{1/2}$};
\node at (-8.5,-1.5) {$\Gamma_{0}$};

\node at (0.5,-1.5) {$\Gamma_{sq}$};
\node at (-4,-1.5) {$\Gamma_{1/2}$};
\node at (-8.5,-1.5) {$\Gamma_{0}$};
\node at (-10.75,0.5) {$\Sigma_{sq}:$};

\end{tikzpicture}

\begin{tikzpicture}
\tikzset{middlearrow/.style={
        decoration={markings,
            mark= at position 0.8 with {\arrow{#1}} ,
        },
        postaction={decorate}
    }
}

\tikzset{
    partial ellipse/.style args={#1:#2:#3}{
        insert path={+ (#1:#3) arc (#1:#2:#3)}
    }
}


\draw[thick,-latex] (0,0) to (0,1);
\draw[thick,-latex] (0,0) to (1,0);
\draw[thick,-latex] (1,1) to (0,1);
\draw[thick,-latex] (1,1) to (1,0);

\draw[thick,-latex] (-1,2) to (0,1);
\draw[thick,-latex] (1,1) to (2,2);
\draw[thick,-latex] (0,0) to (-1,-1);
\draw[thick,-latex] (2,-1) to (1,0);

\draw[thick,-latex] (-10+4.5,2) to [out=-30,in=180] (-8.5+4.5,1.5)  to [out=0,in=210] (-7+4.5,2) ;
\draw[thick,-latex] (-7+4.5,-1) to [out=150,in=0]  (-8.5+4.5,-0.5) to [out=180,in=30](-10+4.5,-1);

\draw[thick,-latex] (-4,0.5) [partial ellipse=180:-180: 0.7cm and 0.7cm];

\draw[thick,-latex] (-10,2) to [out=-30,in=180] (-8.5,1.5)  to [out=0,in=210] (-7,2) ;
\draw[thick,-latex] (-7,-1) to [out=150,in=0]  (-8.5,-0.5) to [out=180,in=30](-10,-1);

\draw[red,dashed]  (0.5,0) [partial ellipse=-180:180: 0.9cm and 0.4cm];
\draw[red,dashed]  (0.5,1) [partial ellipse=-180:180: 0.9cm and 0.4cm];

\draw[red,dashed]  (-4,-0.3) [partial ellipse=-180:180: 0.9cm and 0.5cm];
\draw[red,dashed]  (-4,1.3) [partial ellipse=-180:180: 0.9cm and 0.5cm];

\draw[->] (-2,0.5) to (-1,0.5);
\draw[->] (-2-5+0.25,0.5) to (-1-5+0.25,0.5);

\node at (0.5,-1.5) {$\Gamma_{sq}$};
\node at (-4,-1.5) {$\Gamma_{1/2}'$};
\node at (-8.5,-1.5) {$\Gamma_{0}'$};
\node at (-10.75,0.5) {$\Sigma_{sq}':$};
\end{tikzpicture}
\caption{The square cobordisms.}\label{square-co}
\end{figure}
The square cobordism $\Sigma_{sq}$ from $\Gamma_0$ to $\Gamma_{sq}$ in Figure \ref{square-co} can be described as the composition of two cobordisms. The first one is the
``birth cobordism'' from $\Gamma_0$ to $\Gamma_{1/2}$ which is the disjoint union of the product cobordism $\Gamma_0\times I$ and a disk (thought as a cobordism
from the empty set to a circle). The second one consists of two skein cobordisms in the regions surrounded by the dashed red ellipses 
in Figure \ref{square-co}.
The square cobordism $\Sigma_{sq}'$ can be described similarly.
\begin{LE}\cite{Kh-sl3}*{Proposition 9}\label{square-homology}
Let $\Gamma_0,\Gamma_0', \Gamma_{sq}$ be the oriented planar webs given in Figure \ref{square-co}. Then
\begin{equation*}
  g_{sq}:=F(\Sigma_{sq})+ F(\Sigma'_{sq}) : F(\Gamma_0)\oplus F(\Gamma_0')\to F(\Gamma_{sq})
\end{equation*}
is an isomorphism. Reversing the direction of the cobordisms we obtain $\overline{\Sigma}_{sq}$ and  $\overline{\Sigma}_{sq}'$.
The maps
$F(\overline{\Sigma}_{sq})\circ g_{sq}$ and $F(\overline{\Sigma}_{sq}')\circ g_{sq}$ 
are projections onto the $F(\Gamma_0)$ and $F(\Gamma_0')$ summands respectively.
\end{LE}

\begin{PR}\label{RIVD=D'}
Let $D$ and $D'$ be the two diagrams in Figure \ref{RIVDD'}.
Up to a shifting of the homological and quantum gradings, $F(D)$ and $F(D')$ are chain homotopy equivalent.
\end{PR}
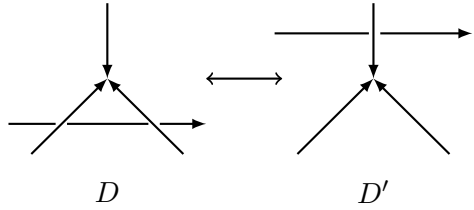
\begin{figure}
\centering
\begin{tikzpicture}
\draw[thick,-latex] (-1,-1) to (0,0); \draw[thick,-latex] (1,-1) to (0,0); \draw[thick,-latex] (0,1)to (0,0);
\draw[thick,-latex, dash pattern=on 0.62cm off 0.15cm on 1.07cm off 0.15cm on 100cm] (-1.3,-0.6) to (1.3,-0.6);

\draw[thick,<->] (1.3,0) to (2.3,0);

\draw[thick,-latex] (2.5,-1) to (3.5,0); \draw[thick,-latex] (4.5,-1) to (3.5,0); \draw[thick,-latex] (3.5,1)to (3.5,0);
\draw[thick,-latex, dash pattern=on 1.24cm off 0.15cm  on 100cm] (2.2,0.6) to (4.8,0.6);

\node at (0,-1.5) {$D$};
\node at (3.5,-1.5) {$D'$};
\end{tikzpicture}
\caption{A Reidemeister move of type IV for an oriented spatial web}\label{RIVDD'}
\end{figure}
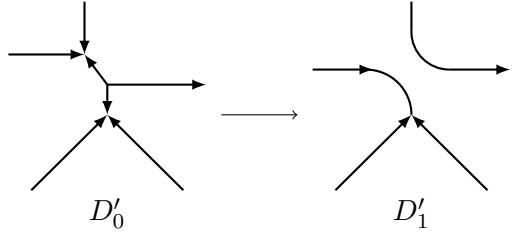
\begin{figure}
\centering
\begin{tikzpicture}

\draw[thick,-latex] (-1.5,-1) to (-0.5,0); \draw[thick,-latex] (0.5,-1) to (-0.5,0); \draw[thick,-latex] (-0.8,1.5)to (-0.8,0.8);
\draw[thick,-latex] (-1.8,0.8) to (-0.8,0.8);  \draw[thick,-latex] (-0.5,0.4) to (0.8,0.4); \draw[thick,-latex] (-0.5,0.4) to (-0.5,0);
\draw[thick,-latex] (-0.5,0.4) to (-0.8,0.8);

\draw[thick,-latex] (2.5,-1) to (3.5,0); \draw[thick,-latex] (4.5,-1) to (3.5,0); \draw[thick] (3.5,1.5)to (3.5,1.1);
\draw[thick,-latex] (2.2,0.6) to (3,0.6);  \draw[thick,-latex] (4,0.6) to (4.8,0.6);

\draw[thick] (2.9,0.6) to [out=0,in=90] (3.5,0);
\draw[thick] (3.5,1.1) to [out=-90,in=180] (4,0.6);

\draw[->] (1,0) to (2,0);

\node at (-0.5,-1.3) {$D_0'$};
\node at (3.5,-1.3) {$D_1'$};

\end{tikzpicture}
\caption{The (1-dimensional) cube of the two resolutions of $D'$}\label{RIV-cube-D'}
\end{figure}
\begin{figure}
\centering
\begin{tikzpicture}
\tikzset{middlearrow/.style={
        decoration={markings,
            mark= at position 0.8 with {\arrow{#1}} ,
        },
        postaction={decorate}
    }
}

\draw[thick,-latex] (0,1)to (0,0);
\draw[thick] (-1,-1) to (-0.7,-0.7); \draw[thick,-latex] (-0.3,-0.3) to (0,0);
\draw[thick] (1,-1) to (0.7,-0.7);   \draw[thick,-latex]  (0.3,-0.3) to (0,0);

\draw[thick] (-1.3,-0.6) to (-0.8,-0.6); \draw[thick] (-0.3, -0.6) to (0.3,-0.6); \draw[thick,-latex] (0.8,-0.6) to (1.3,-0.6);

\draw[thick] (-0.8,-0.6) to [out=0, in=225] (-0.3,-0.3);
\draw[thick] (-0.7,-0.7) to [out=45, in=180] (-0.3, -0.6);
\draw[thick] (0.7,-0.7) to [out=135, in=180] (0.8,-0.6);
\draw[thick] (0.3,-0.6) to [out=0,in=-45] (0.3,-0.3);

\draw[thick,-latex] (4,1)to (4,0);
\draw[thick,middlearrow={latex}] (3,-1) to (3.3,-0.7); \draw[thick,-latex] (3.7,-0.3) to (4,0);
\draw[thick] (5,-1) to (4.7,-0.7);   \draw[thick]  (4.3,-0.3) to (4,0);

\draw[thick,middlearrow={latex}] (2.7,-0.6) to (3.2,-0.6); \draw[thick,middlearrow={latex}] (3.7, -0.6) to (4.3,-0.6); \draw[thick,-latex] (4.8,-0.6) to (5.3,-0.6);

\draw[thick] (4.7,-0.7) to [out=135, in=180] (4.8,-0.6);
\draw[thick] (4.3,-0.6) to [out=0,in=-45] (4.3,-0.3);

\draw[thick] (4-0.8,-0.6) to [out=0,in=45] (4-0.7,-0.7) ;
\draw[thick] (4-0.3,-0.3) to [out=-135,in=180] (4-0.3, -0.6);

\draw[thick,middlearrow={latex}] (3.64, -0.525) to (4-0.7,-0.65);

\draw[thick,-latex] (0,-2)to (0,-3);
\draw[thick] (-1,-4.3) to (-0.7,-4); \draw[thick,-latex] (-0.3,-3.3) to (0,-3);
\draw[thick,middlearrow={latex}] (1,-4.3) to (0.7,-4);   \draw[thick,-latex]  (0.3,-3.3) to (0,-3);

\draw[thick] (-1.3,-3.6) to (-0.8,-3.6); \draw[thick,middlearrow={latex}] (-0.3, -3.9) to (0.3,-3.9); \draw[thick,-latex] (0.8,-3.6) to (1.3,-3.6);

\draw[thick] (-0.8,-3.6) to [out=0, in=225] (-0.3,-3.3);
\draw[thick] (-0.7,-4) to [out=45, in=180] (-0.3, -3.9);

\draw[thick] (0.7,-4) to [out=135, in=180] (0.3,-3.9);
\draw[thick] (0.8,-3.6) to [out=180,in=-45] (0.3,-3.3);

\draw[thick,middlearrow={latex}]    (0.62,-3.57) to (0.57,-3.95);

\draw[thick,-latex] (4,-2)to (4,-3);
\draw[thick,middlearrow={latex}] (3,-4.3) to (4-0.7,-4); \draw[thick,-latex] (4-0.3,-3.3) to (4,-3);
\draw[thick,middlearrow={latex}] (5,-4.3) to (4.7,-4);   \draw[thick,-latex]  (4.3,-3.3) to (4,-3);

\draw[thick,middlearrow={latex}] (2.7,-3.6) to (4-0.8,-3.6); \draw[thick,middlearrow={latex}] (4-0.3, -3.9) to (4.3,-3.9); \draw[thick,-latex] (4.8,-3.6) to (5.3,-3.6);

\draw[thick] (4-0.3, -3.9) to [out=180, in=225] (4-0.3,-3.3);
\draw[thick] (4-0.7,-4) to [out=45, in=0] (4-0.8,-3.6) ;

\draw[thick] (4.7,-4) to [out=135, in=180] (4.3,-3.9);
\draw[thick] (4.8,-3.6) to [out=180,in=-45] (4.3,-3.3);
\draw[thick,middlearrow={latex}]    (4.62,-3.57) to (4.57,-3.95);

\draw[thick,->]   (3.56,-3.75) to (3.368,-3.85);

\draw[->]  (1.5,0) to  (2.5,0);
\draw[->]  (1.5,-3) to (2.5,-3);
\draw[->]  (0,-1) to (0,-1.7);
\draw[->]  (4,-1) to (4,-1.7);

\node at (-1,0.7) {$D_{00}$};
\node at (3,0.7) {$D_{10}$};
\node at (-1,-2.5) {$D_{01}$};
\node at (3,-2.5) {$D_{11}$};

\end{tikzpicture}
\caption{The cube of 4 resolutions of $D$.}\label{RIV-cube-D}
 \end{figure}
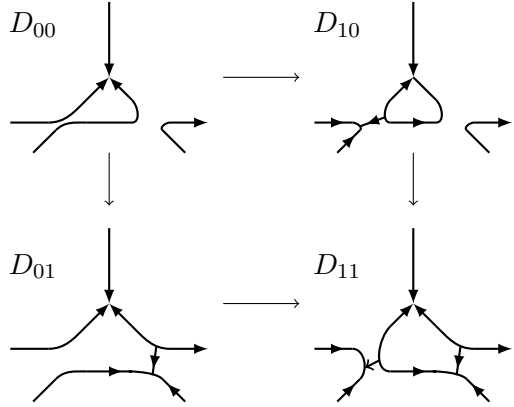
 \begin{figure}
 \begin{tikzpicture}
\tikzset{middlearrow/.style={
        decoration={markings,
            mark= at position 0.8 with {\arrow{#1}} ,
        },
        postaction={decorate}
    }
}

\draw[thick,middlearrow={latex}] (4,1)to (4,0);
\draw[thick,middlearrow={latex}] (3,-1) to (3.3,-0.7);
\draw[thick] (5,-1) to (4.7,-0.7);  

\draw[thick,middlearrow={latex}] (2.7,-0.6) to (3.2,-0.6);  \draw[thick,-latex] (4.8,-0.6) to (5.3,-0.6);

\draw[thick] (4.7,-0.7) to [out=135, in=180] (4.8,-0.6);

\draw[thick] (4-0.8,-0.6) to [out=0,in=45] (4-0.7,-0.7) ;

\draw[thick] (4,0) to [out=-90,in=45] (4-0.7,-0.65);

\node at (4,-1.3) {$D'_{10}$};
\end{tikzpicture}
\caption{Removing the digon from $D_{10}$.}\label{D10'}
\end{figure}
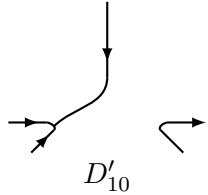
\begin{proof}
By Figure \ref{ind-cube}, we have
\begin{equation}\label{D'-map}
  F(D')=\textbf{Tot}[F(D'_0)\{3\}\to F(D'_1)\{2\}][-1]
\end{equation}
where $D_0'$ and $D_1'$ are the diagrams in Figure \ref{RIV-cube-D'}. Similarly, we have $F(D)$ is the total complex of the 2-dimensional cube
\begin{equation*}
\xymatrix{
  F(D_{00})\{-4\} \ar[d]_{g_1} \ar[r]^{f_1} & F(D_{10})\{-5\}\ar[d]^{g_2} \ar@{=}[r] &F(D'_{10})\{-4\}\oplus F(D'_{10})\{-6\} \\
  F(D_{01})\{-5\} \ar[r]^{f_2} & F(D_{11})\{-6\} \ar@{=}[r] & F(D'_{10})\{-6\}\oplus F(D'_1)\{-6\}
  }
\end{equation*}
where $D_{ij}$ are diagrams in Figure \ref{RIV-cube-D}.

By Lemma \ref{digon-homology}, we can identify $F(D_{10})\{-5\}$ with 
$$
F(D'_{10})\{-4\}\oplus F(D'_{10})\{-6\}
$$ 
where $D'_{10}$ is the diagram in Figure \ref{D10'}.
Notice that $D_{10}'$ can be deformed into $D_{00}$.
 By the proof of Proposition \ref{RVD=D'},
we have the composition of the projection $\pi:F(D_{10})\{-5\}\to F(D'_{10})\{-4\}$ and $f_1$ is the identify map.

By Lemma \ref{square-homology} we have the identification
\begin{equation*}
  F(D_{11})\{-6\}=F(D'_{10})\{-6\}\oplus F(D'_1)\{-6\}
\end{equation*}
The second part of Lemma \ref{digon-homology} tells us that the map induced by the digon cobordism
$\overline{\Sigma}_{d1}: D_{10}'\to D_{10}$ is an embedding
$$F(D_{10}')\{-6\}\cong F(D_{10}')\{-6\}\subset F(D_{10})\{-5\}$$
It is not hard to see that the composition of the cobordisms $D_{10}\to D_{11}$ and $\overline{\Sigma}_{d1}: D_{10}'\to D_{10}$ is exactly the
cobordism $\Sigma_{sq}$ in Figure \ref{square-co}. Therefore by Lemma \ref{square-homology} we have the summand $F(D'_{10})\{-6\}$ of
$F(D_{10})\{-5\}$ maps isomorphically onto the summand $F(D'_{10})\{-6\}$ of $F(D_{11})\{-6\}$ under $g_2$.
Now we define a subcomplex $S$ of the cube $F(D)$ in Figure \ref{RIV-cube-D}
as the direct sum of
\begin{itemize}
  \item $F(D_{00})\{-4\}$,
  \item $F(D'_{10})\{-6\}\subset F(D_{10})\{-5\}$ and $F(D'_{10})\{-6\}\subset F(D_{11})\{-6\}$,
  \item $\Imm (f_1\oplus g_1) \subset  F(D_{10})\{-5\}\oplus F(D_{01})\{-5\}$.
\end{itemize}
Our knowledge on $f_1$ and $g_2$ implies that this subcomplex is null-homotopic and the quotient complex $F(D)/S$ is isomorphic to
\begin{equation}\label{final-map}
  \textbf{Tot}[F(D_{01})\{-5\} \to F(D_1')\{-6\}] [1]
\end{equation}
which is a quotient of the bottom edge in Figure \ref{RIV-cube-D}. Notice that $D_{01}\cong D_0'$. Moreover, using Lemma \ref{square-homology} and
Figure \ref{compo-digon-D'D1}, one can show that the map in \eqref{final-map} and the map in \eqref{D'-map} are induced by the same cobordism.
Hence either two maps are the same or differs by a sign, which implies
\begin{equation*}
  F(D')\simeq _h F(D)\{8\}[-2]
\end{equation*}
This completes the proof.
\end{proof}
We can state our main result for this subsection:
\begin{THE}
Given two diagrams $D$ and $D'$ of an oriented spatial web $\Gamma$, 
the two chain complexes $F(D)$ and $F(D')$ are chain homotopy equivalent up to a shifting of the bi-grading.
In particular, the
 \emph{relatively} bi-graded abelian group
$\mathcal{H}(\Gamma):=H(F(D))$ is a well-defined invariant of the oriented spatial web.
\end{THE}
\begin{rmk}
We lose the absolute bi-grading in the situation of webs. But from the proofs of Propositions \ref{RVD=D'} and \ref{RIVD=D'}, we see that we still have an absolute
$\mathbb{Z}/4$ quantum grading.
\end{rmk}
\subsection{The module structure}\label{module-str}
In this subsection we first review a module structure on $F(\Gamma)$ where $\Gamma$ is an oriented planar web discussed in \cite{Kh-sl3}*{Section 6}. 
Then we generalize
it for oriented spatial webs.

Let $U$ be a circle in $\mathbb{R}^2$, then the pair of pants cobordism equips $F(U)$ with a ring structure. From \cite{Kh-sl3} we have
\begin{equation*}
  F(U)=\mathbb{Z}[X]/ X^3
\end{equation*}
where the (quantum) gradings of $1, X, X^2$ are $-2,0,2$ respectively. So the multiplication action $X: F(U) \to F(U)$ is a degree 2 operator.
Let $D_i$ be a disk with $i$ dots on it. We think $D$ as a cobordism from the empty set
to a circle. The generators of $F(U)$ can be described as
\begin{equation*}
  X^i=F(D_i)
\end{equation*}
More generally, let $\Gamma$ be an oriented planar web. Merging a circle near an edge $e_i$ gives a cobordism from $U\sqcup \Gamma$ to $\Gamma$, which induces
an action $X_i: F(\Gamma)\to F(\Gamma)$. Alternatively, the map $X_i$ can be defined as the map induced by the product cobordism
$\Gamma_i\times I$ with a dot in the interior of $e_i\times I$. Those actions make $F(\Gamma)$ into a $R_\Gamma$ module where $R_\Gamma$ is the
commutative graded ring generated by variables 
$X_i$ 
with relations
\begin{equation}\label{ring1}
  X_i+X_i+X_k=0, X_iX_j + X_j X_k + X_k X_i=0, X_i X_j X_k=0
\end{equation}
whenever $e_i, e_j, e_k$ are incident at a common vertex and
\begin{equation}\label{ring2}
  X_i^3=0
\end{equation}
when $e_i$ is a loop. The degree of $X_i$ is defined to be 2.

Now we want to generalized the above module structure for an oriented spatial web $\Gamma$. The following argument 
is adapted from \cite{Kh-module}.
Let $D$ be a diagram of $\Gamma$ and $p\in \Gamma$ be a point in the interior of
edge $e_i$ disjoint from all the crossings. 
By merging a circle at $p$, we can define a map $X_i$ on each summand of $F(D)$ as before. Put all those map together,
we obtain a chain map
\begin{equation*}
  X_p: F(D)\to F(D)
\end{equation*}
which induces a map of bi-degree  $(0,2)$ on $\mathcal{H}(\Gamma)$. If another diagram $D'$ differs from $D$ by a Reidemeister move, then
there is a chain homotopy equivalence (possibly shifting the bi-grading) between $F(D)$ and $F(D')$. We can also assume the 
move does not touch $p$:
a move of an arc above or under $p$ can be made into a move of the arc across the rest of the plane or $S^2$. Under this assumption,
the chain homotopy equivalence commutes with the $X_p$ action. Therefore 
$$X_i:=X_{p,\ast}: \mathcal{H}(\Gamma)\to  \mathcal{H}(\Gamma)$$
 is a well-defined operator which only depends on the edge $e_i$. 
 The homology 
$\mathcal{H}(\Gamma)$ is a $R_\Gamma$-module where $R_\Gamma$ is defined in the same way as before. Notice that $R_\Gamma$ only depends on the underlying
abstract web of $\Gamma$ and does not depend on the embedding.

We have the following example from \cite{Kh-sl3}*{Section 6}:

\begin{eg} Suppose $\Gamma$ is an oriented planar web which can be reduced into circles by removing digons in Figure \ref{digon-webs}.
Using induction and Lemma \ref{digon-homology}, one can show that $\mathcal{H}(\Gamma)$ is a free $R_\Gamma$-module of rank 1.
\end{eg}

\subsection{Pointed webs}\label{pointed-web}
Suppose $\Gamma$ is an oriented spatial web and 
$$\boldsymbol{\delta}=\{\delta_i|i=1,\cdots,m\}$$ 
is a collection of mark points on $\Gamma\setminus V(\Gamma)$ where
$V(\Gamma)$ is the set of vertices of $\Gamma$. The pair $(\Gamma,\boldsymbol{\delta})$ is called a pointed (oriented spatial) web.
In \cite{BLS}, the ($\mathfrak{sl}(2)$) Khovanov homology for a pointed link is defined. 
We want to define a homological invariant for $(\Gamma,\boldsymbol{\delta})$ by imitating \cite{BLS}.  

Take a diagram $D$ for $\Gamma$ as before and assume $\boldsymbol{\delta}$ is disjoint from the crossings.
Let $\Lambda_{\boldsymbol{\delta}}$ be the exterior algebra generated by formal variables $\{x_i\}$. The bi-degree 
(homological degree and quantum degree) of 
$x_i$ is $(1,2)$.
We define a bi-graded chain complex
\begin{equation}\label{PKh}
C(D,\boldsymbol{\delta}):=\Lambda_{\boldsymbol{\delta}} \otimes_{\mathbb{Z}} F(D)
\end{equation} 
with differential
\begin{equation*}
d_{\boldsymbol{\delta}}(x\otimes y ):=(-1)^{\degg_h(x)}x\otimes d(y) + \sum_{i}x_i\wedge x \otimes X_{\delta_i}(y)
\end{equation*} 
where $d$ is the Khovanov differential on  $F(D)$, $X_{\delta_i}$ is the operator defined in Section \ref{module-str} and
$\degg_h$ denotes the homological degree.  From \eqref{PKh} we know $C(D,\boldsymbol{\delta})$ carries a 
$\Lambda_{\boldsymbol{\delta}}$-module structure. It is not hard to check that the differential on 
$\Lambda_{\boldsymbol{\delta}}$ respects this module structure, hence the homology of $\Lambda_{\boldsymbol{\delta}}$ is
a $\Lambda_{\boldsymbol{\delta}}$-module.

 Alternatively, this chain complex
can be described as the total complex of a cube as the definition of $F(D)$. Index the vertices of a $m$-dimensional cube by
subsets $S\subset \{1,\cdots,m\}$. Assign a copy of $F(D)$ to each vertex and assign the map $X_{\delta_i}$
to the edge $S\to S\cup \{i\} $ where $i\notin S$. After adding plus or minus sign to those edge maps appropriately, 
the total complex becomes 
the chain complex $(C(D,\boldsymbol{\delta}),d_{\boldsymbol{\delta}} )$. Equivalently, 
$(C(D,\boldsymbol{\delta}),d_{\boldsymbol{\delta}} )$ can be defined as the iterated total complex (mapping cone)
 of maps $X_{\delta_i}$. 

\begin{THE}
Given two diagrams $D$ and $D'$ of an pointed oriented spatial web $(\Gamma,\boldsymbol{\delta})$, 
the two chain complexes $C(D,\boldsymbol{\delta})$ and 
$C(D',\boldsymbol{\delta})$ are chain homotopy equivalent (as $\Lambda_{\boldsymbol{\delta}}$-modules) 
up to a shifting of the bi-grading.
In particular, the
 \emph{relatively} bi-graded $\Lambda_{\boldsymbol{\delta}}$-module
$\mathcal{H}(\Gamma,\boldsymbol{\delta}):=H(C(D,\boldsymbol{\delta}))$ 
is a well-defined invariant of $(\Gamma,\boldsymbol{\delta})$.
\end{THE}
\begin{proof}
Suppose $D$ and $D'$ differ by a Reidemeister move. By the discussion in Section \ref{module-str},
we can assume this move does not touch $\boldsymbol{\delta}$ and the diagram
\begin{equation*}
\xymatrix{
  F(D) \ar[d] \ar[r]^{X_{\delta_i}} & F(D) \ar[d] \\
  F(D') \ar[r]^{X_{\delta_i}} & F(D')   }
\end{equation*}  
commutes. The two vertical maps are the chain homotopy equivalence between $F(D)$ and $F(D')$. Therefore the mapping cone
of the two horizontal maps are also chain homotopy equivalent. Since $C(D,\boldsymbol{\delta})$ and
$C(D',\boldsymbol{\delta})$ are defined as the iterated mapping cone of maps $X_{\delta_i}$ ($\delta_i\in \boldsymbol{\delta} $),
iterating the above argument shows $C(D,\boldsymbol{\delta})$ and
$C(D',\boldsymbol{\delta})$ are chain homotopy equivalent.  The $\Lambda_{\boldsymbol{\delta}}$-module structure is just 
the shifting of vertices of the cube, which commutes with the chain homotopy equivalence. 
\end{proof}

We can filter the cube used to define $C(D,\boldsymbol{\delta})$ by the cardinal $|S|$. Then we obtain a spectral sequence
whose $E_1$-page is $2^m$ copies of $\mathcal{H}(\Gamma)$. The differential on the $E_1$-page consists of maps 
$X_i=X_{\delta_i,\ast}:\mathcal{H}(\Gamma) \to \mathcal{H}(\Gamma)$.
The maps $X_i$ can be viewed as elements in $R_\Gamma$. Two elements $X_i$ and $X_j$ are equal if $\delta_i$ and $\delta_j$ lie on
the same edge. Let $s=(X_1,\cdots, X_m)\in R_\Gamma^{\oplus m}$, then the $E_1$-page is just the Koszul complex
$K(s,\mathcal{H}(\Gamma))$.
To be more precise, we first define a chain complex
\begin{equation*}
\xymatrix{
  K(s):=0 \to  R_\Gamma \ar[r]^-{\wedge s} & \Lambda^1 R_\Gamma^{\oplus m} \ar[r]^-{\wedge s} & \Lambda^2 R_\Gamma^{\oplus m}
  \ar[r]^-{\wedge s} & \cdots \ar[r]^-{\wedge s} & \Lambda^m R_\Gamma^{\oplus m}
  \to 0
}
\end{equation*}
Then the Koszul complex is defined as
\begin{equation*}
K(s,\mathcal{H}(\Gamma)):=K(s)\otimes_{R_\Gamma}\mathcal{H}(\Gamma)
\end{equation*}
In summary, we have
\begin{PR}\label{koszul}
Suppose $(\Gamma,\boldsymbol{\delta}=\{\delta_i\}_{1\le i \le m})$ is a pointed oriented spatial web and 
 $s=(X_1,\cdots, X_m)\in R_\Gamma^{\oplus m}$ as above. Then there is a spectral sequence converging to 
 $\mathcal{H}(\Gamma,\boldsymbol{\delta})$ whose $E_1$-page is the Koszul complex $K(s,\mathcal{H}(\Gamma))$.
 In particular, we have
 \begin{equation*}
 \rank_{\mathbb{Z}} H(K(s,\mathcal{H}(\Gamma)))\ge \rank_{\mathbb{Z}} \mathcal{H}(\Gamma,\boldsymbol{\delta})
 \end{equation*} 
\end{PR}
When $\Gamma$ is a planar web, the spectral sequence degenerates on the $E_2$-page and 
\begin{equation*}
 \mathcal{H}(\Gamma,\boldsymbol{\delta}) \cong H(K(s,\mathcal{H}(\Gamma))) 
\end{equation*}

\begin{eg}  Let $\Theta$ be the planar theta graph and $\boldsymbol\delta =\{\delta_1,\delta_2\}$ where
$\delta_1,\delta_2$ lie on two distinct edges. We have   
\begin{equation*}
 \mathcal{H}(\Theta)\cong R_\Theta \cong \mathbb{Z}[X_1,X_2]/ (X_1^2 X_2+X_1 X_2^2, X_1^2+X_2^2+X_1X_2       )
\end{equation*}
The homology $\mathcal{H}(\Theta,\boldsymbol{\delta})$ is defined as the homology of the cube
\begin{equation*}
\xymatrix{
  R_\Theta \ar[d]_{X_2} \ar[r]^{X_1} & R_\Theta \ar[d]^{-X_2} \\
  R_\Theta \ar[r]^{X_1} & R_\Theta   }
\end{equation*}
It is straightforward to check that $\mathcal{H}(\Theta,\boldsymbol{\delta})$ is a free abelian group of rank 4.
\end{eg}

\section{The functor $J$ and the spectral sequence}
The instanton Floer homology for spatial trivalent graphs is introduced by Kronheimer and Mrowka in \cite{KM-jsharp}. In this section we  first review
some results we need from their work. Then we will show that there is a spectral sequence whose $E_2$-page is the $\mathfrak{sl}(3)$ Khovanov module
of a spatial web and which converges to the instanton Floer homology of this spatial web.
 All the discussion in Section \ref{sl3Kh} can be done with $\mathbb{F}$-coefficients
verbatim
and we use $F(\Gamma;\mathbb{F})$ to denote the functor defined in Section \ref{sl3Kh} adapted with  $\mathbb{F}$-coefficients.

\subsection{Instanton Floer homology for webs}\label{functor-I}
Suppose $\Gamma$ is a trivalent graph embedded in an oriented compact three-manifold $Y$. The pair $\check{Y}=(Y,\Gamma)$ 
can be equipped with an orbifold structure 
and an orbifold Riemannian metric
such that
the local stabilizer group $H_x\subset SO(3)$ is $\mathbb{Z}/2$ (if $x$ is not a vertex) or the Klein 4-group (if $x$ is a vertex). 
We call $\check{Y}$ a \emph{bifold} and the orbifold metric on it a bifold metric. 
In most of the situations of this paper, $Y$ will just be $S^3$ (or $\mathbb{R}^3$).

A bifold connection over $\check{Y}$ is an orbifold $SO(3)$ connection whose asymptotic holonomy around each edge of $\Gamma$ is an order $2$
element in $SO(3)$. The underlying orbifold vector bundle of a bifold connection is called a bifold bundle.  
Adding a Hopf link $H$ 
contained in a ball $U_\mu\subset Y$ disjoint from $\Gamma$, we obtain a new bifold $(Y,\Gamma\cup H)$. 
Let $E_\mu\to U_\mu\setminus H$ be the $SO(3)$ bundle whose $w_2$ is dual to a small arc joining the two components of $H$. 
We use $\mathcal{A}$ to denote the space of bifold connections over bifold bundles which are identified with $E_\mu$ on $U_\mu\setminus \Gamma$.  
The gauge group $\mathcal{G}$ consists of $SO(3)$ gauge transformations $g$ such that the restriction 
$g|_{U_\mu\setminus \Gamma}: E_\mu\to E_\mu$ can be lifted into to a determinant-1 gauge transformation. 
The instanton Floer homology
$J^\sharp(Y,\Gamma)$ is defined as the Morse homology (with $\mathbb{F}$-coefficients) of the Chern-Simons functional on the configuration space 
$\mathcal{B}^\sharp(Y,K)=\mathcal{A}/\mathcal{G}$. If $\Gamma$ is a spatial trivalent graph, then it can also be viewed as a graph in $S^3$
and we just write $J^\sharp(\Gamma)$ for $J^\sharp(S^3,\Gamma)$. See \cite{KM-jsharp} for more details of the definition.

There is a variant of $J^\sharp(Y,\Gamma)$ also introduced in \cite{KM-jsharp}: take $U_\mu$ to be the whole manifold $Y$ and repeat 
other parts of the definition, we obtain an invariant $I^\sharp(Y,\Gamma)$.
 We will not need $I^\sharp$ until the next section. The reason to introduce the Hopf link
$H$ with non-trivial $w_2$ in the definition (for $J^\sharp$ or $I^\sharp$) is to avoid reducible connections. 

\begin{DEF}
A \emph{generalized closed foam} $\Sigma$ in $\mathbb{R}^4$ is a 2-dimensional subcomplex decorated with dots such that  each point $x$ in $\Sigma$ has
a neighborhood (in $\Sigma$) which is modelled on one of the following
\begin{itemize}
  \item A smoothly embedded disk;
  \item The product of an interval and the letter ``Y'';
  \item A cone with vertex $x$ and whose base is the complete graph $K_4$ with four vertices.
\end{itemize}
The circles or arcs consist of points of the second type is called seams as before. Points of the third type are called tetrahedral points. The connected components
of the complement of seams and tetrahedral points are called facets. We also require that the dots lie in the facets.
\end{DEF}
We can also define generalized foams with boundary:
\begin{DEF}
A \emph{generalized foam with boundary} is the intersection of $\mathbb{R}^3\times [0,1]$ and a generalized closed foam $\Sigma$ such that
$\mathbb{R}^3\times\{0\}$ and $\mathbb{R}^3\times\{1\}$ are transversal to the facets and seams of $\Sigma$ and contain no tetrahedral points.
In particular,
$\Gamma_1=\mathbb{R}^3\times\{0\}\cap \Sigma$ and $\Gamma_2=\mathbb{R}^3\times\{1\}\cap \Sigma$ are two spatial trivalent graph.
The generalized foam with boundary can be thought as a cobordism from $\Gamma_1$ to $\Gamma_2$.
\end{DEF}
The above definitions are more general than the ones used in Section \ref{sl3Kh}.

Given a cobordism $\Sigma$
from $\Gamma_1$ to $\Gamma_2$, we can also define a 4-dimensional bifold 
$\check{W}:=(S^3\times \mathbb{R}, \Sigma^+\cup H\times \mathbb{R})$ as in the 3-dimensional case.  
Here $\Sigma^+$ is obtained by adding cylindrical ends to $\Sigma$. The concepts of bifold connections, bifold bundles and configuration spaces
can all be defined for the 4-dimensional case. 
If $\Sigma$ has no dots, there is a map $J^\sharp (\Sigma):J^\sharp (\Gamma_1)\to J^\sharp (\Gamma_2)$ defined by counting points in the 
0-dimensional moduli spaces of 
ASD trajectories of bifold connections over $\check{W}$. 
If there are dots on $\Sigma$, the homology class of the meridian of each dot $\delta_i$ determines 
a cohomology class in $H^1(B^\sharp(S^3\times \mathbb{R},\Sigma^+);\mathbb{F})$, which can be represented by a (real) codimension 1 divisor
$V(\delta_i)$ in the configuration space. The map  $J^\sharp (\Sigma)$ is defined by counting points in the 
0-dimensional cutting-down moduli spaces (cutting down by divisors associated with those dots). 
A closed generalized foam $\Sigma$
can be thought a cobordism from an empty graph to another empty graph and $J^\sharp(\Sigma)$ is a number in $\mathbb{F}$.

We summarize some properties of $J^\sharp$:
\begin{PR}\cite{KM-jsharp}\label{jsharp-p}
$J^\sharp$ satisfies the following properties:
\begin{enumerate}
  \item If $\Sigma$ is a closed orientable surface (embedded standardly in $\mathbb{R}^3\subset \mathbb{R}^4$) decorated with dots, then $F(\Sigma)=0$ unless
      \begin{itemize}
        \item $\Sigma$ is a 2-sphere with two dots, then $F(\Sigma)=1$,
        \item $\Sigma$ is a torus without any dots, then $F(\Sigma)=1$.
      \end{itemize}
  \item If $\Sigma_1$ and $\Sigma_2$ are generalized foams in $\mathbb{R}^4$, then $F(\Sigma_1\sqcup \Sigma_2)=F(\Sigma_1)F(\Sigma_2)$.
  \item Do a surgery along a circle inside a facet of $\Sigma$ to obtain three generalized foams $\Sigma_1, \Sigma_2, \Sigma_3$ as in Figure \ref{Sigma123}. Then
        \begin{equation*}
           F(\Sigma)=F(\Sigma_1)+F(\Sigma_2)+F(\Sigma_3)
        \end{equation*}
  \item Let $\Theta(k_1,k_2,k_3)$ be the theta foam (embedded standardly in $\mathbb{R}^3\subset \mathbb{R}^4$)
        with $k_i$ dots on the $i$-th facet (see Figure \ref{theta-foam}), then
        \begin{equation*}
          F(\Theta(k_1,k_2,k_3))=\left\{
                                   \begin{array}{ll}
                                     1, & \hbox{if $(k_1,k_2,k_3)=(0,1,2)$, up to permutation;} \\
                                     0, & \hbox{otherwise.}
                                   \end{array}
                                 \right.
        \end{equation*}
  \item Suppose $\Gamma_1$ and $\Gamma_2$ are two spatial trivalent graphs locally as shown in Figure \ref{digon-webs} (no orientation is needed now), then
       \begin{equation*}
      J^\sharp(\Gamma_1)=J^\sharp(\Gamma_2)\oplus J^\sharp(\Gamma_2)
       \end{equation*}
  \item Suppose $\Gamma_{sq}$, $\Gamma_0$ and $\Gamma_0'$ are spatial trivalent graphs locally as shown in Figure \ref{square-co} 
     (no orientation is needed), then
       \begin{equation*}
      J^\sharp(\Gamma_{sq})=J^\sharp(\Gamma_0)\oplus J^\sharp(\Gamma_0')
       \end{equation*}
  \item Suppose $U$ is the unknot in $\mathbb{R}^3$, then $J^\sharp(U)\cong \mathbb{F}^3$.
\end{enumerate}
\end{PR}
\begin{COR}
With $\mathbb{F}$-coefficients,  $J^\sharp$ satisfies the four axioms in Proposition \ref{4axiom} for oriented foams defined 
in Section \ref{sl3Kh}. Therefore $J^\sharp(\Sigma)=F(\Sigma;\mathbb{F})$ for any closed oriented foam $\Sigma$.
\end{COR}
Given an oriented planar web $\Gamma$, we define
\begin{equation*}
  J'(\Gamma)=\spann \{J^\sharp(\Sigma)|\Sigma\in \Homm_{\text{OF}} (\emptyset,\Gamma)\}\subset J^\sharp(\Gamma)
\end{equation*}
We have a surjective linear map
\begin{equation*}
J'(\Gamma)\to F(\Gamma;\mathbb{F})
\end{equation*}
defined by 
\begin{equation*}
 J^\sharp(\Sigma)\mapsto [\Sigma]
\end{equation*}
where $\Sigma$ is a foam in $\Homm_{\text{OF}} (\emptyset,\Gamma)$.  We need to show this map is well-defined. Suppose
\begin{equation*}
J^\sharp(\Sigma_1)=J^\sharp(\Sigma_2)
\end{equation*}
Then for any $\Phi \in \Homm_{\text{OF}} (\Gamma,\emptyset)$, we have 
\begin{equation*}
J^\sharp(\Phi\circ\Sigma_1)-J^\sharp(\Phi\circ\Sigma_2)=0
\end{equation*}
By the functoriality of $J^\sharp$. Therefore
\begin{equation*}
[\Sigma_1]-[\Sigma_2]=0
\end{equation*}
by the definition of  $F$ in Section \ref{sl3Kh}. So the map is well-defined and $F(\Gamma;\mathbb{F})$ is a subquotient of $J^\sharp(\Gamma)$.
This map turns out to be an isomorphism. 
\begin{PR}\label{J=F}
Given any oriented planar web $\Gamma$, we have
\begin{equation*}  
J^\sharp(\Gamma)=J'(\Gamma)
\end{equation*}
is isomorphic to $F(\Gamma;\mathbb{F})$. The isomorphism is natural  in the sense that given two
oriented planar webs $\Gamma_1$ and $\Gamma_2$, then the digram
\begin{equation*}
\xymatrix{
   J^\sharp(\Gamma_1)\ar[d]_{J^\sharp(\Sigma)} \ar[r]^{\simeq} & F(\Gamma_1,\mathbb{F}) \ar[d]^{F(\Sigma,\mathbb{F})} \\
  J^\sharp(\Gamma_2)\ar[r]^{\simeq} & F(\Gamma_2,\mathbb{F})   } 
\end{equation*}
commutes for any $\Sigma\in \Homm_{\emph{OF}} (\Gamma_1,\Gamma_2)$.
\end{PR}
\begin{proof}
Any oriented planar web can be reduced into a collection of circles by removing digons 
(Figure \ref{digon-webs}) and squares (Figure \ref{square-co})
(cf. \cite{Kh-sl3}*{Section 2}). Using Lemma \ref{digon-homology}, Lemma \ref{square-homology} 
and Parts (5) (6) (7) of Proposition \ref{jsharp-p} 
to do inductions,
it is easy to show that
\begin{equation*}  
\dim_{\mathbb{F}}J^\sharp(\Gamma)=\dim_{\mathbb{F}}F(\Gamma;\mathbb{F})
\end{equation*}
for any oriented planar web $\Gamma$. Therefore we have
\begin{equation*}  
J^\sharp(\Gamma)=J'(\Gamma)\cong F(\Gamma;\mathbb{F})
\end{equation*}
The naturality of the isomorphism follows directly from its definition. 
\end{proof}
Given any (abstract)  trivalent graph $\Gamma$ with edges $\{e_i\}$, we can define a ring  $\mathcal{R}_\Gamma$ generated by $X_i$ over $\mathbb{F}$ modulo
relations \eqref{ring1} and \eqref{ring2}. Alternatively, we can define
\begin{equation*}
\mathcal{R}_\Gamma:=R_\Gamma\otimes_\mathbb{Z} \mathbb{F}
\end{equation*}
where $R_\Gamma$ is defined in Section \ref{module-str}. Now $F(\Gamma;\mathbb{F})$ carries a $\mathcal{R}_\Gamma$-module structure.

Suppose $\Gamma$ is a spatial trivalent graph, then $J^\sharp(\Gamma)$ is also equipped with a module structure in \cite{KM-jsharp}. 
The module structure can be defined in a similar way as in Section \ref{module-str}: given an edge $e_i$ of $\Gamma$, we form a 
generalized foam $\Sigma=\Gamma\times I$ with a dot on $e_i\times I$ and define
\begin{equation*}
X_i=J^\sharp(\Sigma):J^\sharp(\Gamma)\to J^\sharp(\Gamma)
\end{equation*}
This makes $J^\sharp(\Gamma)$ into a $\mathcal{R}_\Gamma$-module. The isomorphism in Proposition \ref{J=F} respects the module structures 
on both sides.

\subsection{Exact triangles}
From now on 
we assume certain perturbations are chosen so that all the moduli spaces are regular.
Denote the three links in Figure \ref{+01} by $L_2, L_0, L_1$ and the three links in Figure \ref{-01} by $L_2', L_0', L_1'$. 
The Floer chain complex used to define $J^\sharp(L_i)$ (or $J^\sharp(L_i')$) is denoted by $C_i$ ($i\in \mathbb{Z}/3)$.
The following result can 
be read from \cite{KM-triangle}:
\begin{PR}\label{3-gon}
There exist maps
\begin{equation*}
f_{i,{i+k}} :C_{i}\to C_{i+k}, ~k=1,2~\text{or}~3
\end{equation*}
such that
\begin{eqnarray*}
  df_{i,i+1}+f_{i,i+1}d&=& 0 \\
  df_{i,i+2}+f_{i,i+2}d &=&  f_{i+1,i+2}f_{i,i+1}\\
  df_{i,i+3}+f_{i,i+3}d &=& f_{i+1,i+3}f_{i,i+1} + f_{i+2,i+3}f_{i,i+2}+g_i
\end{eqnarray*}
where $d$ is the Floer differential and 
$$g_i:C_{i}\to C_{i+3}=C_i$$
 is a quasi-isomorphism. 
\end{PR}
By \cite{OS-ss}*{Lemma 4.2}, the above proposition implies
\begin{COR}\label{3-gon-quasi-iso}
The map
\begin{equation*}
f_{{i},i+2}+ f_{{i+1},i+2} :  \cone (f_{{i},i+1})=C_{i}\oplus C_{i+1}\to C_{i+2}
\end{equation*}
is an quasi-isomorphism. 
\end{COR}
There is a standard foam cobordism $\Sigma_{i,i+1}$ from $L_i$ to $L_{i+1}$ (or from $L_i'$ to $L_{i+1}'$). 
In particular, $\Sigma_{0,1}$ is the cobordism in Figure \ref{skein-co}. 
We can make $\Sigma_{i,i+1}$ into a surface $\Sigma_{i,i+1}^+$ with cylindrical end and $\Sigma_{i,i+1}$ is a product away from a four-ball where
the skein move happens. The pair 
$(S^3\times \mathbb{R},\Sigma_{i,i+1}^+\cup H\times \mathbb{R})$ is equipped with a family of bifold metrics parametrized by
$G_{i,i+1}\cong \mathbb{R}$. The $\mathbb{R}$-translation on $S^3\times \mathbb{R}$ gives an $\mathbb{R}$-action on $G_{i,i+1}$ and the quotient
space $\breve{G}_{i,i+1}$ consists of a single metric.  
The map $f_{i,i+1}$ is defined by counting points in the zero-dimensional moduli space over $\breve{G}_{i,i+1}$. Passing to homology, 
$f_{i,i+1}$ is exactly the map $J^\sharp(\Sigma_{i,i+1}):J^\sharp(L_i)\to J^\sharp(L_{i+1})$.

In general, let $\Sigma_{i,i+k}:L_i\to L_{i+k}$ be the composition of cobordisms 
$\Sigma_{i+k-1,i+k},\cdots, \Sigma_{i,i+1}$ and $\Sigma_{i,i+k}^+$ be the surface obtained by adding cylindrical ends to $\Sigma_{i,i+k}$.
A $(k-1)$-dimensional family of bifold metrics on $(S^3\times \mathbb{R},\Sigma_{i,i+k}^+)$
parametrized by $\breve{G}_{i,i+k}$ is defined in \cite{KM-jsharp} for $k=1,2,3$
(see also \cite{KM:Kh-unknot}). Again  there is a $k$-dimensional family of metrics parametrized 
by $G_{i,i+k}$ with an $\mathbb{R}$-action such that $\breve{G}_{i,i+k}=G_{i,i+k}/\mathbb{R}$.  The map $f_{i,i+k}$ is defined by counting points
in the zero-dimensional moduli spaces over $\breve{G}_{i,i+k}$ and the relations in Proposition \ref{3-gon} are derived by analyzing 
the boundaries of 1-dimensional moduli spaces over $\breve{G}_{i,i+k}$.   

\subsection{The spectral sequence}\label{J-ss}
In this subsection we will prove the following.
\begin{THE}\label{s-sequence}
Let $\Gamma$ be an oriented spatial web. There is a spectral sequence of $\mathcal{R}_\Gamma$-modules whose $E_2$-page is the $\mathfrak{sl}(3)$
Khovanov module $\mathcal{H}(\Gamma;\mathbb{F})$ and which converges to $J^\sharp(\Gamma)$.
\end{THE}
The proof is very little different from the proof of a corresponding result for 
$\mathfrak{sl}(2)$ Khovanov homology in \cite{KM:Kh-unknot}. The current situation is
even simpler in some aspect: 
since we are working in characteristic 2, there is no need to deal with the orientations and signs as in \cite{KM:Kh-unknot}.  
All the necessary ingredients of the proof are already included in \cite{KM:Kh-unknot,KM-jsharp,KM-triangle}. 

Suppose $\Gamma$ is an oriented spatial web and $D$ is a diagram for $\Gamma$ with $n$ crossings. Given $v\in \{0,1\}$, let $D_v$ be the 
planar web obtained by resolving the crossings using $v$ as in Section \ref{sl3Kh}. 
We define
\begin{eqnarray*}
  |v|_1 &=& \sum_i |v_i| \\
  |v|_\infty &=& \sup_i |v_i|
\end{eqnarray*}
for any $v\in \mathbb{R}^n$.
Given any $v\in \{0,1\}^n$, we use $C_v$ to denote the Floer chain complex for $J^\sharp(D_v)$ and
use $d_v$ to denote the differential on $C_v$. Given $v\le u$ in $\{0,1\}^c$ (i.e. $v_i\le u_i$ for all $i$ where $v_i, u_i$ are coordinates of $v,u$), 
there is a cobordism $S_{vu}$ which can be made into a surface
$S_{vu}^+\subset S^3\times \mathbb{R}$ 
with cylindrical end in the standard way. 
The surface
$S_{vu}^+$ is a product surface away from $|v-u|_1$ four-balls where the skein moves happen.
By shifting these four-balls containing the skein moves, we can define a family of metrics parameterized by $G_{vu}\cong \mathbb{R}^{|v-u|_1}$.
There is an $\mathbb{R}$-action on $G_{vu}$ defined by the $\mathbb{R}$-translation on $\mathbb{R}\times S^1\times S^2$. The quotient $G_{vu}\slash \mathbb{R}$ is denoted by
$\breve{G}_{vu}$. $G_{vu}$ and $\breve{G}_{vu}$ are not compact in general but can be compactified into spaces
$G_{vu}^+$ and $\breve{G}_{vu}^+$ by adding broken metrics.
Let
$
  M_{vu}(\alpha,\beta)_d
$
be the $d$-dimensional moduli space of ASD trajectories over $G_{vu}$
with limiting connection $\alpha$ on the incoming end and $\beta$ on the outgoing end. Here $\alpha$ and $\beta$ are generators for
$C_v$ and $C_u$ respectively.
There is an obvious map
$M_{vu}\to G_{vu}$ and the $\mathbb{R}$-action on $G_{vu}$ can be lifted on $M_{vu}(\alpha, \beta)_d$. We denote the quotient by
\begin{equation*}
  \breve{M}_{vu}(\alpha,\beta)_{d-1}:=M_{vu}(\alpha, \beta)_d/\mathbb{R}
\end{equation*}
Both $M_{vu}(\alpha,\beta)_d$ and $\breve{M}_{vu}(\alpha,\beta)_{d-1}$ can be partially compactified by adding broken trajectories lying over broken metrics in
$\partial G_{vu}^+$ and $\partial\breve{G}_{vu}^+$. These are only partial compactifications because of the possible appearance of bubbles. 
We denote these partial compactifications by $M^+_{vu}(\alpha,\beta)$ and $\breve{M}^+_{vu}(\alpha,\beta)$ respectively.
A group homomorphism
\begin{equation}\label{fvu}
  f_{vu}: C_v\to C_u
\end{equation}
can be defined by counting  points in the 0-dimensional moduli space:
\begin{equation}\label{fvu=}
  f_{vu}(\alpha):=\sum_\beta \# \breve{M}_{vu}(\alpha,\beta)_0 \cdot \beta
\end{equation}
where $\beta$ runs through all the generators for $C_u$. In the case $v=u$, $m_{vv}$ is just the Floer differential.

Suppose $e_i$ is an edge of $\Gamma$ and $\delta_i\in e_i$ is a point disjoint with all the three-balls
where the skein moves of $\Gamma$ happen. The homology class of the meridian around $e_i$ determines an element $\zeta_i$ in 
$H^1(B^\sharp(Y,\Gamma);\mathbb{F})$. 
Let $\nu(\delta_i)\subset S^3$ be a suitable neighborhood of $\delta_i$ that is also disjoint from  all the skein moves and 
contains $H$.  
A (real) codimension 1 divisor 
$$V(\delta_i)\subset B^\sharp(\delta_i):= B^\sharp( \nu(\delta_i)\times (-1,1),(\Gamma\cap\nu(\delta_i))\times (-1,1))$$ 
is defined in \cite{KM-jsharp}*{Section 3.5}.
The pullback of $V(\delta_i)$ to $B^\sharp(Y,\Gamma)$ is the dual of $\zeta_i$.
Given $v\le u$ as before, we define a map
\begin{equation*}
  r_{vu}:C_v\to C_u
\end{equation*}
by
\begin{equation}\label{rvu}
 r_{vu}(\alpha):= \sum_\beta \#  ({M}_{vu}(\alpha,\beta)_1 \cap V(\delta_i)) \cdot \beta
\end{equation}
where $({M}_{vu}(\alpha,\beta)_1 \cap V(\delta_i))$ should be understood as pulling back the divisor $V(\delta_i)$ 
by the restriction $r: M_{vu}\to B^\sharp(\delta_i)$. We assume the divisors
$V(\delta_i)$ are generic so that all the intersections are regular. When $v=u$,
 $r_{vv}:C_v\to C_v$ induces the map 
 $$X_i=J^\sharp(\dot{D}_v\times I):J^\sharp(D_v)\to J^\sharp(D_v)$$ 
 on the homology where $\dot{D}_v\times I$ is the foam $D_v\times I$ with a dot on the facet $e_i\times I$.
Now we define
\begin{equation}\label{CF}
\mathbf{C}=\bigoplus_{v\in \{0,1\}^n}C_v,~ \mathbf{F}:=\sum_{v\le u}f_{vu}:\mathbf{C}  \to \mathbf{C}
\end{equation} 
and
\begin{equation}\label{Rrvu}
 \mathbf{R}:=\sum_{v\le u}r_{vu}:\mathbf{C}  \to \mathbf{C}
\end{equation} 

\begin{PR}\label{CF-chain}
We have 
\begin{equation*}
 \mathbf{F}\mathbf{F}=0, ~ \mathbf{F}\mathbf{R}+\mathbf{R}\mathbf{F}=0
\end{equation*}
so that 
\begin{equation*}
 (\mathbf{C},\mathbf{F})
\end{equation*}
is a chain complex and 
\begin{equation*}
 \mathbf{R}:\mathbf{C}\to \mathbf{C}
\end{equation*}
is a chain map.
\end{PR}
\begin{proof}
The two equalities are derived by counting boundary points of 1-dimensional moduli spaces $\breve{M}_{vu}^+(\alpha,\beta)_1$ 
and ${M}_{vu}^+(\alpha,\beta)_2 \cap V(\delta_i)$ respectively. 
The boundary of $\breve{M}_{vu}^+(\alpha,\beta)_1$ consists of  
$$\breve{M}_{vw}(\alpha,\eta)_0  \times \breve{M}_{wu}(\eta,\beta)_0$$
where $v\le w\le u$. The moduli space
$\breve{M}_{vu}^+(\alpha,\beta)_1$ may have open ends coming from bubbles. But
the number of such ends is always an even number by the argument in \cite{KM-jsharp}*{Section 3.3}. Since we are working in characteristic 2,
we can ignore such ends. Therefore we have
\begin{equation*}
\sum_{w,\eta}\# \breve{M}_{vw}(\alpha,\eta)_0 \cdot \# \breve{M}_{wu}(\eta,\beta)_0=0
\end{equation*}
This implies the component of $\mathbf{F}\mathbf{F}$ mapping $C_v$ to $C_u$ is $0$. Since $v,u$ are arbitrary, this completes the proof
of the first equality. 

The boundary of ${M}_{vu}^+(\alpha,\beta)_2 \cap V(\delta_i)$ consists of
\begin{equation*}
 ({M}_{vw}(\alpha,\eta)_1\cap V(\delta_i))  \times \breve{M}_{wu}(\eta,\beta)_0
\end{equation*}
and
\begin{equation*}
\breve{M}_{vw}(\alpha,\eta)_0  \times  ({M}_{wu}(\eta,\beta)_1\cap V(\delta_i))
\end{equation*}
where $v\le w\le u$. We still need to exclude the possible appearance of bubbles.
According to \cite{KM-jsharp}*{Section 3.3},
a codimension-2 bubble can only arise in the situation that $v=u, \alpha=\beta$ and there is a sequence
of connections $A_l\in M_{vv}(\alpha,\alpha)_2$ such that
\begin{equation*}
\xymatrix{
  A_l|_{S^3\times \mathbb{R}\setminus z} \ar[r] & \widetilde{\alpha}|_{S^3\times \mathbb{R}\setminus z}
  } 
\end{equation*} 
where $\widetilde{\alpha}$ represents the product trajectory obtained by pulling $\alpha$ back to the product cobordism and
$z$ is the bubble point on the \emph{seam} of the orbifold points. Since $\delta_i$ is disjoint from the seam and $V(\delta_i)$ is generic,
a sequence of connections in ${M}_{vu}^+(\alpha,\beta)_2 \cap V(\delta_i)$ can never converge to an ideal connection 
$(\widetilde{\alpha},z)$ as above. Therefore we have
\begin{equation}\label{M-FR+RF}
\begin{aligned}
\sum_{w,\eta}\# ({M}_{vw}(\alpha,\eta)_1\cap V(\delta_i))\cdot \# \breve{M}_{wu}(\eta,\beta)_0 + \\
\sum_{w,\eta}\#\breve{M}_{vw}(\alpha,\eta)_0  \cdot \# ({M}_{wu}(\eta,\beta)_1\cap V(\delta_i)) =0
 \end{aligned}
\end{equation}
This implies the component of $\mathbf{F}\mathbf{R}+\mathbf{R}\mathbf{F}$ mapping $C_v$ to $C_u$ is $0$.
Since $v,u$ are arbitrary, this completes the proof of the second equality. 

See \cite{KM:Kh-unknot} for more details on the counting argument. Also compare the proof of 
Proposition \ref{CU-chain} where a more general situation is discussed.
\end{proof}

\begin{PR}\label{cube-quasi-iso}
The chain complex 
\begin{equation*}
 (\mathbf{C},\mathbf{F})
\end{equation*}
is quasi-isomorphic to the Floer chain complex $C(\Gamma)$ for $J^\sharp(\Gamma)$. Moreover, the induced isomorphism 
\begin{equation*}
 H(\mathbf{C})\cong H(C(\Gamma))=J^\sharp(\Gamma)
\end{equation*}
intertwines with the induced action $\mathbf{R}_\ast$ and $X_i$.
\end{PR}
\begin{proof}[Sketch of the proof]
If there is only one crossing in the diagram $D$, then the quasi-isomorphism is given in Corollary \ref{3-gon-quasi-iso}. The general case can be obtained by
``iterating''  Corollary \ref{3-gon-quasi-iso}. Let $v\in\{0,1,2\}^n$, we can resolve $\Gamma$ to obtain a spatial web $D_{v}$ where a ``2-resolution''
means  keeping the crossing without any change. If $v,u \in\{0,1,2\}^n$ and $v\le u$, we still have a cobordism $S_{vu}:D_v\to D_u$ obtained
by composing the skein cobordisms. 
For $v\le u$ and $|u-v|_\infty\le 1$ in $\{0,1,2\}^n$, we can define a family of metrics on $(S^3\times \mathbb{R},S_{vu}^+)$
by shifting the four-balls containing the skein move as before. Let  $C_v$ be the Floer chain complex for $J^\sharp({D_v})$.
Form a chain complex 
\begin{equation*}
(\mathbf{C}_2,\mathbf{F}_2)=(\bigoplus_{v'\in\{0,1\}^{n-1}}C_{2v'},\mathbf{F}_2) 
\end{equation*}  
equipped with an action $\mathbf{R}_2$ by \eqref{fvu=} and \eqref{rvu} as in the definitions of $\mathbf{F}$ and $\mathbf{R}$. 

Suppose $v', u'\in\{0,1\}^{n-1}$ and $v'\le u'$, then a family of metrics parametrized by 
$$G_{0v',2u'}\cong \mathbb{R}^{|2u'-0v|_1}=\mathbb{R}^2\times \mathbb{R}^{|u'-v'|_1}$$  
can be defined on $(S^3\times \mathbb{R}, \Sigma_{0v',2u'}^+)$.  The $\mathbb{R}^{|u'-v'|_1}$ component is obtained by shifting four-ball containing
the skein moves for crossings associated with $v',u'$. Take $(x,y)\in \mathbb{R}^2$. When $y-x\ge 1$, they describe the locations 
of the two skein moves of the first crossing. When $y-x < 1$, the metric is defined by stretching along the collar of a specific 3-dimensional 
sub-bifold of $(S^3\times \mathbb{R}, \Sigma_{0v',2u'}^+)$. See \cite{KM-triangle}*{Section 6} and \cite{KM:Kh-unknot}*{Section 7.2} for
the precise definition. 

Suppose $v\in\{0,1\}^{n}, u'\in\{0,1\}^{n-1}$ and $v\le 2u'$.
Counting points in the 0-dimensional moduli spaces $\breve{M}_{v,2u'}$ over 
$\breve{G}_{v,2u'}:=G_{v,2u'}/\mathbb{R}$ gives a map $\mathbf{H}$ from $\mathbf{C}$ to $\mathbf{C}_2$. 
To be more precise, we define
\begin{equation*}
\mathbf{H}:=\sum h_{v,2u'}
\end{equation*}
where
\begin{equation*}
 h_{v,2u'}:C_v\to C_{2u'}
\end{equation*}
is defined by
\begin{equation*}
 h_{v,2u'}(\alpha)=\sum_\beta \# \breve{M}_{v,2u'}(\alpha,\beta)_0 \cdot \beta
\end{equation*}
The arguments in \cite{KM:Kh-unknot}*{Section 7}
and \cite{KM-triangle} show $\mathbf{H}$ is a quasi-isomorphism. 

We define
\begin{equation*}
\mathbf{R}':\mathbf{C} \to \mathbf{C}_2, ~\mathbf{R}':=\sum r'_{v,2u'}
\end{equation*} 
where
\begin{equation*}
r'_{v,2u'}:C_v\to C_{2u'}
\end{equation*} 
is defined by
\begin{equation*}
r'_{v,2u'}(\alpha):=\sum_\beta \#  ({M}_{v,2u'}(\alpha,\beta)_1 \cap V(\delta_i)) \cdot \beta
\end{equation*} 
Counting boundary points of the 1-dimensional cutting-down moduli space 
$M^+_{v,2u'}\cap V(\delta_i)$, we obtain an equality similar to \eqref{M-FR+RF} (with $u$ replaced by $2u'$), which implies
\begin{equation*}
 \mathbf{R}_2\mathbf{H}+\mathbf{H}\mathbf{R}+\mathbf{F}_2\mathbf{R}'+\mathbf{R}'\mathbf{F}=0
\end{equation*}
This means $\mathbf{H}\mathbf{R}+\mathbf{R}_2\mathbf{H}$ is chain homotopic to zero. 
Hence $\mathbf{H}$ induces an isomorphism on the homologies that intertwines with $\mathbf{R}$ and $\mathbf{R}_2$.  

Iterating the above argument, we finally obtain a quasi-isomorphism between $\mathbf{C}$ and $\mathbf{C}_{2,2,\cdots,2}=C(\Gamma)$. 
Moreover, the induced isomorphism on
homologies intertwines with $\mathbf{R}_\ast$ and  
$\mathbf{R}_{{2,\cdots,2},\ast} =X_i:J^\sharp(\Gamma)\to J^\sharp(\Gamma)$. 
\end{proof}
Now we are ready to prove the main result of this section:
\begin{proof}[Proof of Theorem \ref{s-sequence}]
Using Proposition \ref{cube-quasi-iso} and filtering the cube $\mathbf{C}$ by the sum of coordinates, we obtain a spectral sequence
which converges to $J^\sharp(\Gamma)$ and whose $E_1$-page is
\begin{equation*}
\bigoplus_{v\in \{0,1\}^n} J^\sharp(D_v)
\end{equation*}
The differential on the $E_1$-page is 
\begin{equation*}
\sum_{\substack{v<u\\|v-u|_1=1}} J^\sharp(\Sigma_{vu})
\end{equation*}
When $v,u\in \{0,1\}^n$ only differ at one coordinate, $\Sigma_{vu}$ is just the skein cobordism in Figure \ref{skein-co}. 
By Proposition \ref{J=F} and the definition of the Khovanov homology in Section \ref{sl3Kh}, it is clear that the $E_2$-page is exactly
$\mathcal{H}(\Gamma;\mathbb{F})$. 

On the $E_1$-page, the operator $\mathbf{R}$ becomes the operator $X_i:J^\sharp(D_v)\to J^\sharp(D_v)$. By the discussion at the end of
Section \ref{functor-I}, 
we know that  $\mathbf{R}$ on the $E_2$-page is the operator $X_i$ on $\mathcal{H}(\Gamma;\mathbb{F})$. This completes the
proof.
\end{proof}
\begin{rmk}
Suppose $L$ is a link in $\mathbb{R}^3$ with components $L_1,\cdots, L_m$. Let $R_m$ be the ring
$$\mathbb{Z}[X_1,\cdots, X_m]/(X_1^2,\cdots, X_m^2)$$
The $\mathfrak{sl}(2)$ Khovanov homology of $L$ can be equipped with an $R_m$-module structure \cite{Kh-module}. 
An $R_m$-module structure on the instanton Floer homology $I^\sharp(L;\mathbb{Z})$ is defined implicitly in \cite{KM:Kh-unknot}*{Section 8.3}. 
 The argument in this section can be used to show that the spectral sequence in 
\cite{KM:Kh-unknot}
relating the $\mathfrak{sl}(2)$ Khovanov homology to $I^\sharp(L;\mathbb{Z})$ respects the $R_m$-module structures. 
\end{rmk}

\section{The functor $I$ and the spectral sequence}
In this section we will construct a spectral sequence relating the pointed $\mathfrak{sl}(3)$ Khovanov homology
to the functor $I^\sharp$ 
which we briefly reviewed at the beginning of Section \ref{functor-I}. 

Let $\Gamma$ be a \emph{connected} spatial trivalent graph. The main difference between the definitions of $J^\sharp(\Gamma)$ and $I^\sharp(\Gamma)$ 
is the choice of gauge groups: we use the gauge group $\mathcal{G}$
of determinant-1 over a ball containing $H$ for $J^\sharp$ and the gauge group $\widetilde{\mathcal{G}}$ of
determinant-1 over the whole manifold $S^3$ for $I^\sharp$. Since the definition of $J^\sharp$ uses a larger gauge group, 
we have a (normal) covering map
\begin{equation*}
 \widetilde{B}^\sharp(S^3,\Gamma) \to B^\sharp(S^3,\Gamma)
\end{equation*}
between the configuration spaces for $I^\sharp$ and $J^\sharp$. The deck transformation group is 
${\mathcal{G}}/\widetilde{\mathcal{G}}\cong H^1(S^3\setminus \Gamma;\mathbb{F})$. 
In particular each element in $H_1(S^3\setminus \Gamma;\mathbb{F})$ determines a cohomology class in 
$H^1({B}^\sharp(S^3,\Gamma);\mathbb{F})$ which is already used in Section \ref{J-ss}.

Given a spatial trivalent graph $\Gamma$, let $C(\Gamma)$ be the Floer chain complex for $J^\sharp(L)$. 
Pick edges $e_1,e_2,\cdots, e_m$ such that the meridians around those edges form a basis for $H_1(\mathbb{R}^3\setminus \Gamma, \mathbb{F})$.
Use $\zeta_i$ to denote the cohomology class in $H^1({B}^\sharp(S^3,\Gamma);\mathbb{F})$ determined by $e_i$.
Pick a point $\delta_i\in e_i$ for each edge $e_i$ and use $V(\delta_i)$ to denote the (real) divisor representing $\zeta_i$ used in 
Section \ref{J-ss}.  Suppose $S\subset \{1,\cdots,m\}$, we define a map
\begin{equation*}
U_S: C(\Gamma)\to C(\Gamma)
\end{equation*}
by
\begin{equation}\label{U_S}
U_S(\alpha):=\sum_\beta \sum_{[\gamma]\in \breve{M}(\alpha,\beta)_0 } (\prod_{i\in S} \sharp ([\gamma]\cap V(\delta_i))) \cdot \beta
\end{equation}
We need to explain the notation in the above formula. Recall $\breve{M}(\alpha,\beta)_0={M}(\alpha,\beta)_1/\mathbb{R}$
is the 0-dimensional unparameterized moduli space with limiting connections $\alpha$ and $\beta$ on the ends, which is the quotient of the
1-dimensional moduli space of trajectories by the translation action. In the formula, $[\gamma]$ as an element in $\breve{M}(\alpha,\beta)_0$ 
represents a translation equivalent class in ${M}(\alpha,\beta)_1$, which is isomorphic to a real line. In the intersection 
$[\gamma]\cap V(\delta_i)$, $[\gamma]$ is viewed as a connected component of  ${M}(\alpha,\beta)_1$. 
Notice that
since we are working with $\mathbb{F}$-coefficients, only the parity matters in the counting. The map
$U_\emptyset$ is nothing but the Floer differential. The map $U_{\{i\}}:C(\Gamma)\to C(\Gamma)$ induces 
$X_i:J^\sharp(\Gamma)\to J^\sharp(\Gamma)$ which is the map used to define the 
$\mathcal{R}_\Gamma$-module structure on $J^\sharp(\Gamma)$ in Section \ref{functor-I}.

\begin{PR}\cite{KM-jsharp}*{Section 8.5}\label{J-I}
Let $(\widetilde{C}(\Gamma),\mathcal{D})$ be the chain complex defined by
$$
\widetilde{C}(\Gamma)=\bigoplus_{S\subset \{1,\cdots,m\}} C_S(\Gamma),~ \mathcal{D}=\sum_{A\subset B\subset \{1,\cdots,m\}} \mathcal{D}^{A,B} 
$$
where each $C_S(\Gamma)$ is a copy of $C(\Gamma)$ and $\mathcal{D}^{A,B}:C_A\to C_B$ is $U_{B\setminus A}$. Then 
\begin{equation*}
H(\widetilde{C}(\Gamma),\mathcal{D})\cong I^\sharp(\Gamma)
\end{equation*}
Therefore by filtering $\widetilde{C}(\Gamma)$ by the cardinal of $S$, we obtain a spectral sequence whose $E_1$-page is $2^m$ copies of 
$J^\sharp(\Gamma)$ and which converges to $I^\sharp(\Gamma)$.
\end{PR}
From this proposition, we obtain an analogue of Proposition \ref{koszul}.
\begin{COR}\label{koszul-J-I}
Let $\Gamma$ and $\{e_1,\cdots,e_m\}$ be given as above.  
There is a spectral sequence converging to $I^\sharp(\Gamma)$ whose $E_1$-page is the Koszul complex 
$$
K(s,J^\sharp(\Gamma))
$$
where $s=(X_1,\cdots,X_m)\in \mathcal{R}_\Gamma ^{\oplus m}$. In particular, we have
\begin{equation*}
\rank_\mathbb{F} H(K(s,J^\sharp(\Gamma)) )\ge \rank_\mathbb{F}I^\sharp(\Gamma)) 
\end{equation*}
\end{COR}

Next we want to combine Proposition \ref{J-I} and Proposition \ref{cube-quasi-iso} to derive a new spectral sequence. 
Now assume $\Gamma$ is a connected oriented spatial web and a diagram $D$ with $n$ crossings for $\Gamma$ is chosen. This is exactly 
the same situation as in Section \ref{J-ss} except the additional requirement on the connectedness of $\Gamma$.  All the constructions
for $\Gamma$ in Section \ref{J-ss} still work and we will continue using the 
notation from there. Assume the the neighborhood $\nu(\delta_i)$ of point $\delta_i\in e_i$ used to define $V(\delta_i)$ 
is disjoint from all the skein moves as before. 
We define
\begin{equation*}
\widetilde{\mathbf{C}}:=\bigoplus_{S\subset \{1,\cdots,m\}}\mathbf{C}_S
\end{equation*}
where $\mathbf{C}_S$ is a copy of $\mathbf{C}$. 
Suppose $v\le u$ in $\{0,1\}^n$ and $S\subset \{1,\cdots,m\} $, we define a map
\begin{equation*}
U_{S,vu}: C_v \to C_{u}
\end{equation*}
by
\begin{equation}\label{USvu}
U_{S,vu}(\alpha):=\sum_\beta \sum_{[\gamma]\in \breve{M}_{vu}(\alpha,\beta)_0 } (\prod_{i\in S} \sharp ([\gamma]\cap V(\delta_i))) \cdot \beta
\end{equation}
This definition is similar to \eqref{U_S} except that we are working on moduli spaces over a family of metrics now. 
When $v=u$, $U_{S,vv}$ is just the map $U_S:C(D_v)\to C(D_v)$.  When $S=\emptyset$, $U_{\emptyset,vu}$ is just the map
$f_{vu}$ in \eqref{fvu}.

Suppose $A\subset B\subset \{1,\cdots,m\}$ and $v\le u$ in $\{0,1\}^n$. Let
$C_{A,v}$ and $C_{B,u}$ be the copies of $C_v$ and $C_u$ in $\mathbf{C}_A$ and $\mathbf{C}_B$ respectively. We define a map
\begin{equation*}
\mathcal{D}^{A,B}_{vu}: C_{A,v} \to C_{B,u}
\end{equation*}
by setting
\begin{equation*}
\mathcal{D}^{A,B}_{vu}= U_{B\setminus A, vu}
\end{equation*}
Now we define
\begin{equation*}
\widetilde{\mathbf{F}}:\widetilde{\mathbf{C}}\to \widetilde{\mathbf{C}}
\end{equation*}
by 
\begin{equation*}
\widetilde{\mathbf{F}}:= \sum_{\substack{A\subset B  \\ v\le u}} \mathcal{D}^{A,B}_{vu}
\end{equation*}
Take a component $\mathbf{C}_S$ of $\widetilde{\mathbf{C}}$, the part of $\widetilde{\mathbf{F}}$ which maps $\mathbf{C}_S$ to itself is the map
$\mathbf{F}$ in \eqref{CF}. 

\begin{PR}\label{CU-chain}
The pair $(\widetilde{\mathbf{C}},\widetilde{\mathbf{F}})$ is a chain complex, i.e.
\begin{equation*}
\widetilde{\mathbf{F}}\widetilde{\mathbf{F}}=0
\end{equation*}
\end{PR}
\begin{proof}
The proof is very similar to the proof of Proposition \ref{CF-chain}: the equality in the proposition is obtained by counting the boundary points of
1-dimensional moduli spaces. The current situation is a little bit more subtle, so we will do the counting argument more carefully here.

Take $A\subset B\subset \{1,\cdots,m\}$ and  $v\le u$ in $\{0,1\}^n$. Without loss of generality, we assume $A=\emptyset$
and $B=S$. We want to show that the component of  $\widetilde{\mathbf{F}}\widetilde{\mathbf{F}}$ mapping 
$\mathbf{C}_{\emptyset,v}$ to $\mathbf{C}_{S,u}$ is $0$. When $S=\emptyset$, 
 this is exactly the content of Proposition \ref{CF-chain}. So we assume $S\neq 0$ from now on.

The 1-dimensional boundary strata of $M_{vu}^+(\alpha,\beta)_2$ consist of 
\begin{enumerate}
  \item The strata corresponding to trajectories sliding off the incoming end of the cobordism, having the form
        \begin{equation*}
        \breve{M}_{vw}(\alpha,\eta)_0  \times M_{wu}^+(\eta,\beta)_1
        \end{equation*}
        where $v\le w \le u$.
  \item The strata corresponding to trajectories sliding off the ongoing end of the cobordism, having the form
        \begin{equation*}
          {M}_{vw}(\alpha,\eta)_1  \times \breve{M}_{wu}^+(\eta,\beta)_0
        \end{equation*}
        where $v\le w \le u$.
\end{enumerate}
Let $N\subset M_{vu}(\alpha,\beta)_2$ be a connected component and $\breve{N}:=N/\mathbb{R}$ is non-compact. As a non-compact 1-manifold
$\breve{N}$ must be an open interval. There are two possible sources for the ends of $\breve{N}$: the broken trajectories and
codimension-2 bubbles. By  \cite{KM-jsharp}*{Section 3.3}, 
a codimension-2 bubble can only arise in the situation that $v=u, \alpha=\beta$ and there is a sequence
of connections $A_l\in M_{vv}(\alpha,\alpha)$ such that
\begin{equation*}
\xymatrix{
  A_l|_{S^3\times \mathbb{R}\setminus z} \ar[r] & \widetilde{\alpha}|_{S^3\times \mathbb{R}\setminus z}
  } 
\end{equation*} 
where $\widetilde{\alpha}$ represents the product trajectory obtained by pulling $\alpha$ back to the product cobordism and
$z$ is the bubble point on the \emph{seam} of the orbifold points. Since $\delta_i$ is disjoint from the seam and $V(\delta_i)$ is generic,
the ends from bubbling have no contribution to $\overline{N}\cap V(\delta_i)$. 
Let $N^+\subset  M_{vu}^+(\alpha,\beta)_2$ be the partial compactification of $N$ obtained by adding broken trajectories. 
Then $N^+\cap V(\delta_i)$ ($i\in S$) is a compact 1-manifold possibly with boundary. If $\breve{N}$ has
only one end from broken trajectory 
$$([\gamma_1],[\gamma_2])\in \breve{M}_{vw}(\alpha,\eta)_0  \times \breve{M}_{wu}^+(\eta,\beta)_0$$ 
then $N$ has two 1-dimensional boundary strata
$$
\{[\gamma_1]\}\times [\gamma_2] \subset \breve{M}_{vw}(\alpha,\eta)_0  \times M_{wu}^+(\eta,\beta)_1,
$$
and
$$
[\gamma_1]\times \{[\gamma_2]\} \subset {M}_{vw}(\alpha,\eta)_1  \times \breve{M}_{wu}^+(\eta,\beta)_0.
$$
Since $N^+\cap V(\delta_i)$ is a compact 1-manifold with boundaries $[\gamma_1]\cap V(\delta_i)$ and $[\gamma_2]\cap V(\delta_i)$, we have
\begin{equation}\label{1end}
\#[\gamma_1]\cap V(\delta_i)+\#[\gamma_2]\cap V(\delta_i)=0
\end{equation}
Similarly, if both ends of $\breve{N}^+$ are broken trajectories 
$$([\gamma_1],[\gamma_2])~ \text{and} ~([\gamma_1'],[\gamma_2'])$$
then 
\begin{equation}\label{2end}
\#[\gamma_1]\cap V(\delta_i)+\#[\gamma_2]\cap V(\delta_i)=\#[\gamma_1']\cap V(\delta_i)+\#[\gamma_2']\cap V(\delta_i)
\end{equation}
Now \eqref{1end} implies
\begin{align}
0&=\prod_{i\in S} [\#[\gamma_1]\cap V(\delta_i)+\#[\gamma_2]\cap V(\delta_i) ] \nonumber \\
0&=\sum_{S_1\sqcup S_2=S}  [\prod_{i\in S_1}\#[\gamma_1]\cap V(\delta_i)  \prod_{j\in S_2}\#[\gamma_2]\cap V(\delta_j)]  \label{T1}
\end{align}
and \eqref{2end} implies
\begin{equation*}
\prod_{i\in S} [\#[\gamma_1]\cap V(\delta_i)+\#[\gamma_2]\cap V(\delta_i)] =\prod_{i\in S} [\#[\gamma_2']\cap V(\delta_i)+\#[\gamma_2']\cap V(\delta_i) ] 
\end{equation*}
which is the same as 
\begin{align}
\sum_{S_1\sqcup S_2=S}  [\prod_{i\in S_1}\#[\gamma_1]\cap V(\delta_i)  \prod_{j\in S_2}\#[\gamma_2]\cap V(\delta_j)] = \nonumber
\\
\sum_{S_1\sqcup S_2=S}  [\prod_{i\in S_1}\#[\gamma_1']\cap V(\delta_i)  \prod_{j\in S_2}\#[\gamma_2']\cap V(\delta_j)]  \label{T2}
\end{align}
Take the sum of \eqref{T1} or \eqref{T2} through all the connected components $N$, we obtain
\begin{equation}\label{FF=0}
\sum_{v\le w\le u} \sum_{S_1\sqcup S_2=S}  \langle U_{S_2,wu}\circ U_{S_1,vw}(\alpha), \beta \rangle =0 
\end{equation}
where $ \langle a, \beta \rangle$ denotes the coefficient of the generator $\beta$ in $a$. By the definition of $\widetilde{\mathbf{F}}$, \eqref{FF=0}
says 
the component of  $\widetilde{\mathbf{F}}\widetilde{\mathbf{F}}$ mapping 
$\mathbf{C}_{\emptyset,v}$ to $\mathbf{C}_{S,u}$ is $0$. This completes the proof.
\end{proof}

\begin{PR}\label{I-cube-iso}
The chain complex $(\widetilde{\mathbf{C}},\widetilde{\mathbf{F}})$ is quasi-isomorphic to $(\widetilde{C}(\Gamma),\mathcal{D})$ defined 
in Proposition \ref{J-I}. 
\end{PR}
\begin{proof}
The proof is similar to the proof of Proposition \ref{cube-quasi-iso} and we will continue using the notation from that proof.
A quasi-isomorphism 
$$
\mathbf{H}:(\mathbf{C},\mathbf{F})\to (\mathbf{C}_2,\mathbf{F}_2)
$$
is constructed in the proof of Proposition \ref{cube-quasi-iso} by counting points in the 0-dimensional moduli spaces
$\breve{M}_{v,2u'}$ where $v\in \{0,1\}^n, u' \in \{0,1\}^{n-1}$. 

Formula \eqref{USvu} can be used to define a map
$$
U_{S,vu}: C_v\to C_u
$$
for $v\le u$ in $\{0,1,2\}^n$. 
Let
$$
\widetilde{\mathbf{C}}_2:=\bigoplus_{S\subset\{1,\cdots m\}} \mathbf{C}_{2,S}
$$
where each $\mathbf{C}_{2,S}$ is a copy of $\mathbf{C}_2$.
Then a differential $\widetilde{\mathbf{F}}_2$ on $\widetilde{\mathbf{C}}_2$ can be define using maps $U_{S,2v',2u'}$ where $v'\le u'$
in $\{0,1,2\}^{m-1}$ and imitating the definition of $\widetilde{\mathbf{F}}$. Moreover, we can define a map
$$
\widetilde{\mathbf{H}}: \widetilde{\mathbf{C}}\to \widetilde{\mathbf{C}}
$$
whose component from $\mathbf{C}_{A}$ to $\mathbf{C}_{2,B}$ ($A\subset B$) is $U_{B\setminus A, v, 2u'}$ where
$v\in \{0,1\}^m, u'\in \{0,1\}^{m-1}$  and $v\le 2u'$. When $A=B$, this map is just the quasi-isomorphism $\mathbf{H}$.

Now counting the boundary points in the moduli space $M_{v,2u'}^+(\alpha,\beta)_2\cap V(\delta_i)$ as in the proof
of Proposition \ref{CU-chain} shows that $\widetilde{\mathbf{H}}$ is a chain map. Both 
$\widetilde{\mathbf{C}}$ and $\widetilde{\mathbf{C}}_2$ can be filtered by the cardinal $|S|$ and 
$\widetilde{\mathbf{H}}$ respects the filtrations. 
The $E_1$-pages
of the associated spectral sequences are $2^m$ copies of $H(\mathbf{C})$ and $\mathbf{C}_2$. On the $E_1$-pages,
the  induced map $\widetilde{\mathbf{H}}_\ast$ is $2^m$-copies of ${\mathbf{H}}_\ast$, which is an isomorphism.
Therefore $\widetilde{\mathbf{H}}$ is a quasi-isomorphism. 

Iterating the above argument, we finally obtain a quasi-isomorphism from $\widetilde{\mathbf{C}}$ to 
$\widetilde{\mathbf{C}}_{2,\cdots,2}=\widetilde{C}(L)$.
\end{proof}
The main theorem of this section is
\begin{THE}\label{I-ss}
Suppose $\Gamma$ is a connected oriented spatial web and ${\boldsymbol{\delta}}=\{\delta_i\}$ is a collection of 
points in the interior of edges of $\Gamma$ such that the homology classes of meridians around $\delta_i$ form a basis
of $H_1(S^3\setminus \Gamma;\mathbb{F})$. Then there is a spectral sequence whose $E_2$-page is 
$\mathcal{H}(\Gamma,{\boldsymbol{\delta}};\mathbb{F})$ and which converges to $I^\sharp(\Gamma)$.
\end{THE}
\begin{proof}
By Propositions \ref{J-I} and \ref{I-cube-iso}, we have $H(\widetilde{\mathbf{C}})\cong I^\sharp(\Gamma)$. Recall that
\begin{equation*}
\widetilde{\mathbf{C}}=\bigoplus_{v,S}C_{S,v}
\end{equation*} 
We can filter $\widetilde{\mathbf{C}}$ by $|v|_1+|S|$. The $E_1$-page of the associated spectral spectral sequence is 
exactly the chain complex
\begin{equation*}
(C(D,{{\boldsymbol{\delta}}})\otimes_{\mathbb{Z}}\mathbb{F},d_{{\boldsymbol{\delta}}} )
\end{equation*} 
used to define $\mathcal{H}(\Gamma,{\boldsymbol{\delta}};\mathbb{F})$ in Section \ref{pointed-web}. Therefore the $E_2$-page is 
$\mathcal{H}(\Gamma,{\boldsymbol{\delta}};\mathbb{F})$.
\end{proof}
\begin{rmk}
A $\boldsymbol{\Lambda}_{\boldsymbol{\delta}}$-module structure
is defined on $\mathcal{H}(\Gamma,{\boldsymbol{\delta}};\mathbb{F})$ in Section \ref{pointed-web} (adapted with $\mathbb{F}$-coefficients)
where $\boldsymbol{\Lambda}_{\boldsymbol{\delta}}=\Lambda^\ast[x_1,\cdots,x_m]\otimes_{\mathbb{Z}}\mathbb{F}$. 
We can equip $\widetilde{\mathbf{C}}$ with a $\boldsymbol{\Lambda}_{\boldsymbol{\delta}}$-module structure
by requiring
\begin{equation*}
x_i:\mathbf{C}_S\to \mathbf{C}_{S\cup\{i\}}
\end{equation*}
is the identification map where $i\notin S$ (recall that both $\mathbf{C}_S$ and  $\mathbf{C}_{S\cup\{i\}}$ are copies of $\mathbf{C}$).
A similar module structure can be defined on $\tilde{C}(\Gamma)$ (hence on $I^\sharp(\Gamma)$)
and the quasi-isomorphism in Proposition \ref{I-cube-iso}
respects the module structures. Therefore the spectral sequence in Theorem \ref{I-ss} respects this module structure 
and the $E_2$-page is isomorphic to $\mathcal{H}(\Gamma,{\boldsymbol{\delta}};\mathbb{F})$ as 
$\boldsymbol{\Lambda}_{\boldsymbol{\delta}}$-modules. A priori the 
$\boldsymbol{\Lambda}_{\boldsymbol{\delta}}$-module structure on $I^\sharp(\Gamma)$ 
relies on the specific isomorphism in Proposition \ref{J-I}. Indeed this module structure is intrinsic. There is 
a  ${\mathcal{G}}/\widetilde{\mathcal{G}}\cong H^1(S^3\setminus \Gamma; \mathbb{F})$ 
action on $I^\sharp(\Gamma)$. Let 
$\{y_i\}\subset H^1(S^3\setminus \Gamma; \mathbb{F})$ be the dual basis of 
$\{[m_i]\}\subset H_1(S^3\setminus \Gamma; \mathbb{F})$ where $m_i$ is the meridian around $\delta_i$. 
Using this dual basis, we can describe the the group ring 
\begin{equation*}
\mathbb{F}H^1(S^3\setminus \Gamma; \mathbb{F}) \cong \mathbb{F}[y_1,\cdots,y_m]/(y_i^2-1|i=1,\cdots,m)
\end{equation*}
where we view $H^1(S^3\setminus \Gamma; \mathbb{F})$ as a multiplicative group. We have an isomorphism 
\begin{align*}
\boldsymbol{\Lambda}_{\boldsymbol{\delta}}=\Lambda^\ast[x_1,\cdots,x_m]\otimes_{\mathbb{Z}}\mathbb{F}
&\to \mathbb{F}[y_1,\cdots,y_m]/(y_i^2-1|i=1,\cdots,m)  \\
x_i&\mapsto 1+y_i
\end{align*}
\end{rmk}

\section{The detection of the planar theta graph}
In this section, we will prove Theorem \ref{theta-detection}.

For  a general trivalent graph $\Gamma$ in a 3-manifold $Y$, 
$I^\sharp(Y,\Gamma)$ is defined with $\mathbb{F}$-coefficients. If $L$ is a link,
then $I^\sharp(Y,L)$ can be defined with $\mathbb{Z}$-coefficients. This is exactly the situation in \cite{KM:Kh-unknot}. In this case, 
we use  $I^\sharp(Y,L;R)$ to denote the instanton Floer homology with $R$-coefficients 
($R=\mathbb{Z},\mathbb{Q},\mathbb{C}$). If we do not specify the coefficients, $I(Y,L)$ always means the Floer homology with 
$\mathbb{F}$-coefficients.

From now on, we take
$\Gamma$ to be a spatial theta web.
\begin{PR}\label{H<4}
Suppose $\delta_1,\delta_2$ are two mark points lying on two distinct edge of $\Gamma$.
If
\begin{equation*}
\mathcal{H}(\Gamma;\mathbb{F})\cong  \mathcal{H}(\Theta;\mathbb{F})
\end{equation*}
as $\mathcal{R}_\Theta$-modules, then 
\begin{equation*}
\rank_{\mathbb{F}} \mathcal{H}(\Gamma,\{\delta_1,\delta_2\};\mathbb{F})\le 4
\end{equation*}
\end{PR}
\begin{proof}
This follows from Proposition \ref{koszul} and the example at the end of Section \ref{pointed-web} (adapted with $\mathbb{F}$-coefficients).
\end{proof}
We also have a parallel proposition for $J^\sharp$ and $I^\sharp$.
\begin{PR}\label{I<4}
Suppose $\delta_1,\delta_2$ are two mark points lying on two distinct edge of $\Gamma$.
If
\begin{equation*}
J^\sharp(\Gamma)\cong  J^\sharp(\Theta) (\cong \mathcal{H}(\Theta;\mathbb{F}))
\end{equation*}
as $\mathcal{R}_\Theta$-modules, then 
\begin{equation*}
\rank_{\mathbb{F}} I^\sharp(\Gamma)\le 4
\end{equation*}
\end{PR}
\begin{proof}
This follows from Proposition \ref{koszul-J-I} and the example at the end of Section \ref{pointed-web} (adapted with $\mathbb{F}$-coefficients).
\end{proof}

Recall that 
\begin{equation*}
I^\sharp(\Gamma):=I(\Gamma \cup H)
\end{equation*}
where $H$ is a Hopf link together with an arc joining the two components representing the $w_2$.
Let $B_1$ and $B_2$ be two small ball neighborhood of the two vertices of $\Gamma$ so that the boundary sphere
$S_1$ (or $S_2$) meets $\Gamma$ transversally at three points. Pick a diffeomorphism 
$h:(S_1,\Gamma\cap S_1)\to (S_2,S_2\cap \Gamma)$ such that 
that $\Gamma\cap S_1$ and $\Gamma\cap S_2$ are identified by the three arcs $\Gamma\cap (S^3\setminus B_1\cup B_2)$.
Move $H$ into $B_1$, cut $S^3$ along $S_1$ and $S_2$ and re-glue using $h$, we obtain 
\begin{equation*}
(S^1\times S^2, L) \cup (S^3,\Theta\cup H)
\end{equation*}
where $L$ is a link with three components in $S^1\times S^2$.
We use $S$ to denote the sphere in $S^1\times S^2$ obtained by identifying $S_1$ and $S_2$. Since $L\cap S$ consists of 
three points ,the non-integral condition (\cite{KM:Kh-unknot}*{Definition 3.1}) is satisfied. Hence the instanton Floer homology
$I(S^1\times S^2,L;\mathbb{Z})$ is well-defined (without adding $H$).
\begin{PR}\label{L=2gamma}
Let $\Gamma$ and $L$ be given as above, then we have
\begin{equation*}
I(S^1\times S^2,L)\otimes I^\sharp(\Theta) \cong I^\sharp(\Gamma)\oplus I^\sharp(\Gamma)
\end{equation*}
\end{PR}
\begin{proof}
This is exactly the content of the proof of \cite{KM-jsharp}*{Proposition 7.5}, which follows from  an
excision theorem.
\end{proof}

\begin{PR}\label{I-theta}
We have
\begin{equation*}
 I^\sharp(\Theta) \cong \mathbb{F}^4
\end{equation*}
\end{PR}
\begin{proof}
We first use an argument which was used in the proof of \cite{KM-deform}*{Proposition 2.1} to replace $\Theta$ by a Hopf link.
Recall that
\begin{equation*}
I^\sharp(\Theta):=I(S^3,\Theta\cup H)
\end{equation*}
The Excision Theorems in \cite{KM-jsharp}*{Section 4} can be used to show that
\begin{equation*}
 I(S^3,\Theta\cup H_1)\otimes I(S^3,\Theta\cup H_2)\cong  I(S^3,\Theta\cup H_1\cup H_2)
 \cong I(S^3,\Theta\cup H_2)\otimes I(S^3,H_1\cup H_2)
\end{equation*}
where $H_1$ and $H_2$ are two disjoint copies of $H$. Since 
$$ 
I(S^3,\Theta\cup H_2)=I^\sharp(\Theta)\neq 0
$$ 
by \cite{KM-jsharp}*{Section 7.1},
 we have 
\begin{equation}\label{theta=H}
 \rank_\mathbb{F} I^\sharp(\Theta)=\rank_\mathbb{F} I(S^3,\Theta\cup H_1)=
 \rank_\mathbb{F} I(S^3,H_1\cup H_2)
\end{equation}
The Chern-Simons functional on the pair $(S^3,H_1\cup H_2)$ has a Morse-Bott critical set $SO(3)$. After a 
suitable generic perturbation, the restriction $F|_{SO(3)}$ of
the perturbed Chern-Simons functional 
becomes a standard Morse function on $SO(3)$ with critical points $\alpha_i$ ($0\le i \le 3$)
and the critical points of $F$ are exactly $\{\alpha_i\}$. The degree of $\alpha_i$ is just $i$ and the moduli space
of trajectories $M(\alpha_i,\alpha_{i-1})_1$ ($1\le i \le 3$) approximates the Morse trajectories on $SO(3)$. Therefore
the differential on $\alpha_i$ ($1\le i \le 3$)  is the same as the differential in the Morse homology of $SO(3)$.
Since we only have a relative $\mathbb{Z}/4$-grading on those critical points, the differential on
$\alpha_0$ may be non-zero. Now we have a cyclic chain complex 
\begin{equation*}
\xymatrix{
  \mathbb{Z}\{\alpha_3\}  \ar[r]^{0} & \mathbb{Z}\{\alpha_2\} \ar[d]^{2} \\
  \mathbb{Z}\{\alpha_0\} \ar[u]_{d} & \mathbb{Z}\{\alpha_1\} \ar[l]^{0} 
}
\end{equation*}
If $d\neq 0$, then we have $I(S^3,H_1\cup H_2;\mathbb{Q})=0$. 

Let $Y_1$ and $Y_2$ be two homology 3-spheres.
Fukaya's connected sum theorem \cite{Fukaya-sum} relates the instanton Floer homologies 
$I(Y_1;\mathbb{Z}),I(Y_2;\mathbb{Z})$ to $I(Y_1\# Y_2;\mathbb{Z})$ by a 
spectral sequence. An elaboration of Fukaya's theorem with $\mathbb{Q}$-coefficients
can be found in \cite{Don:YM-Floer}*{Section 7.4}. The same problem is 
also studied in \cite{Li-sum}. We want to apply Fukaya's theorem to the connected sum 
\begin{equation*}
(S^3,H_1\cup H_2)=(S^3,H_1)\#(S^3,H_2)
\end{equation*} 
Even though this is not
a connected sum of homology 3-spheres with empty links in Fukaya's setting, the same proof still works. Our situation is
even easier because we are working in an admissible case so that reducible critical points of the Chern-Simons functional 
do not enter into the discussion.  

From \cite{Don:YM-Floer}*{Section 7.4}, we have a spectral sequence which converges to $I(S^1\times S^2, H_1\cup H_2)$ and
whose last possibly non-degenerate page is 
\begin{align*}
\xymatrix{
  I(S^3, H_1;\mathbb{Q})\otimes I(S^3,  H_2;\mathbb{Q}) \ar[r]^{f} & 
  I(S^3, H_1;\mathbb{Q})\otimes I(S^3,  H_2;\mathbb{Q})   }
\end{align*}
where the differential $f$ is $2\mu(x_1)\otimes 1- 1\otimes 2\mu (x_2)$ ($x_1\in S^3\setminus H_1$ and $x_2\in S^3\setminus H_2$). 
We use
$1$ for the identify map and
$\mu(x)$ for the point operator of degree 4. Our convention for $\mu$ is from \cite{DK}. 

The instanton Floer homology $I(S^3,  H;\mathbb{Q})$ is 1-dimensional. Therefore
the point operator $\mu(x)$ on  $I(S^3,  H;\mathbb{Q})$ is just a scalar product. 
This implies the differential $f$ is zero. 
We obtain $I(S^3,H_1\cup H_2;\mathbb{Q})=\mathbb{Q}^{2}$, which is a contradiction.
Therefore $d$ must be $0$. We have 
\begin{equation*}
I(S^3,H_1\cup H_2;\mathbb{Z})=\mathbb{Z}^2\oplus \mathbb{Z}/2,~~I(S^3,H_1\cup H_2)= \mathbb{F}^{4}
\end{equation*}
By \eqref{theta=H}, the proof is complete.
\end{proof}

When we define $L$, we first remove two balls $B_1$ and $B_2$ from $S^3$. We can define
a tangle $T:=\Gamma\cap (S^3\setminus B_1\cup B_2)$ in $(S^3\setminus B_1\cup B_2)\cong I\times S^2$. We may assume
$T$ is disjoint with $I\times \{\infty\}\subset I\times S^2$ and view $T$ as a tangle in $I\times D^2$ by removing a tubular 
neighborhood of $I\times \{\infty\}$ in $I\times S^2$. Let $S\subset S^1\times S^2$ be a $S^2$-slice as before.
In \cite{Street}, an operator $\muu (S)$ of degree $2$ on $I(S^1\times S^2,L;\mathbb{C})$ is defined and the eigenvalues of
$\muu(S)$ are shown to be $\pm 1$. 
As a degree $2$ operator on a $\mathbb{Z}/4$-graded vector space 
$I(S^1\times S^2,L;\mathbb{C})$, the generalized eigenspaces with eigenvalue $1$ or $-1$ must have the same dimension. 
The generalized eigenspace with eigenvalue $1$ is defined to be the odd tangle Floer homology $\THI^{\text{odd}}(T)$
in \cite{Street}*{Definition 3.3.2}. From the definition we have
\begin{equation}\label{THI-dim}
\rank_{\mathbb{C}} I(S^1\times S^2,L;\mathbb{C}) = 2\rank_{\mathbb{C}} \THI^{\text{odd}}(T)
\end{equation}
The odd tangle Floer homology $\THI^{\text{odd}}(T)$ satisfies the following.
\begin{THE}\cite{AHI}*{Theorem 3.10}
The  Floer homology  $\THI^{\text{odd}}(T)$ is 1-dimensional
 if and only if $T$ is isotopic to a braid. 
\end{THE}

\begin{PR}\label{I-detection-braid}
Let $\Gamma$ and $T$ be given as above. If 
\begin{equation*}
\rank_{\mathbb{F}}I^\sharp(\Gamma)\le 4
\end{equation*}
then we have
\begin{equation*}
\THI^{\emph{odd}}(T)\cong \mathbb{C}
\end{equation*}
Therefore $T$ must be a braid. 
\end{PR}
\begin{proof}
By Proposition \ref{L=2gamma} and Proposition \ref{I-theta}, we have
$$
\rank_{\mathbb{F}} I^\sharp (S^1\times S^2, L) \le 2
$$
Equality \eqref{THI-dim} and the universal coefficients theorem imply 
$$
\rank_{\mathbb{C}} \THI^{\text{odd}}(T) \le 1
$$
Since $\THI^{\text{odd}}(T)$ is odd-dimensional \cite{AHI}*{Proposition 4.10}, its dimension must be 1.
\end{proof}

The diagram of a spatial trivalent graph can be changed by type V moves in Figure \ref{RV}, which changes $T$ by a generator of
the braid group. After finitely many type V moves, $T$ becomes the trivial braid.
 So we have
\begin{COR}\label{I-detection-theta}
Let $\Gamma$ be given as above. If 
\begin{equation*}
\rank_{\mathbb{F}} I^\sharp(\Gamma)\le 4
\end{equation*}
then  $\Gamma$ must be the planar theta graph $\Theta$. 
\end{COR}

Now we are ready to prove the detection theorem.
\begin{proof}[Proof of Theorem \ref{theta-detection}]
By Theorem \ref{I-ss}, Proposition \ref{H<4} and Proposition \ref{I<4}, any one in (b) (c) (d) (e) implies 
$\rank_{\mathbb{F}}I^\sharp(\Gamma)\le 4$. Therefore $\Gamma$ is the planar theta graph by Corollary \ref{I-detection-theta}.
We obtain any one in (b) (c) (d) (e) implies (a).

It is clear that (a) implies (b) (c) and (d). The last term $(e)$ is just Proposition \ref{I-theta}.

\end{proof}

\bibliography{references}
\bibliographystyle{hplain}

\Address

\end{document}